\documentclass[11pt,a4paper,twoside]{amsart}
\usepackage{amssymb,amsmath,amsthm,graphicx,color}
\theoremstyle{plain}
\newtheorem{theorem}{Theorem}[section]
\newtheorem{proposition}[theorem]{Proposition}
\newtheorem{definition}[theorem]{Definition}
\newtheorem{lemma}[theorem]{Lemma}
\newtheorem{corollary}[theorem]{Corollary}
\newtheorem{assumption}[theorem]{Assumption}
\theoremstyle{remark}
\newtheorem{remark}[theorem]{Remark}

\usepackage
[bookmarks=true,
   bookmarksnumbered=false,
   bookmarkstype=toc]
{hyperref}
\allowdisplaybreaks[4]
\numberwithin{equation}{section}

\newcommand{\R}{\mathbb{R}}
\newcommand{\Z}{\mathbb{Z}}

\newcommand{\F}{\mathcal{F}}

\newcommand{\I}{\infty}
\newcommand{\abs}[1]{\left\lvert #1\right\rvert}

\newcommand{\norm}[1]{\left\lVert #1\right\rVert}

\newcommand{\hMn}[4]{\left\lVert #1 \right\rVert_{\hat{M}^{#2}_{{#3},{#4}}}}

\newcommand{\tnorm}[1]{\lVert #1\rVert}

\newcommand{\IN}{\quad\text{in }}

\newcommand{\EQ}[1]{\begin{equation} #1 \end{equation}}
\newcommand{\ALN}[1]{\begin{align*} #1 \end{align*}}

\newcommand{\ALNd}[1]{\begin{aligned} #1 \end{aligned}}

\def\Sch{{\mathcal S}} 
\def\({\left(}
\def\){\right)}
\def\<{\left\langle}
\def\>{\right\rangle}
\def\le{\leqslant}
\def\ge{\geqslant}

\def\d{{\partial}}

\newcommand{\si}{\sigma}

\newcommand{\eps}{\varepsilon}

\DeclareMathOperator{\supp}{supp}

\DeclareMathOperator{\dist}{dist}

\newcommand{\rre}{\mathbb{R}}
\newcommand{\pt}{\partial}

\begin{document}
\title[Refinement of Strichartz estimate for Airy equation]
{Refinement of Strichartz estimate for\\
Airy equation in non-diagonal case\\ 
and its application}
\author[S.Masaki and J.Segata]
{Satoshi Masaki  
and Jun-ichi Segata}
\address{Department systems innovation\\
Graduate school of Engineering Science\\
Osaka University\\
Toyonaka Osaka, 560-8531, Japan}
\email{masaki@sigmath.es.osaka-u.ac.jp}

\address{Mathematical Institute, Tohoku University\\
6-3, Aoba, Aramaki, Aoba-ku, Sendai 980-8578, Japan}
\email{segata@m.tohoku.ac.jp}



\maketitle
\vskip-15mm
\begin{abstract}
In this paper, we give an improvement of the non-diagonal Strichartz estimate 
for Airy equation by using a Morrey type space. 
As its applications, we prove the small data scattering and existence 
of a special non-scattering solutions, which are minimal in suitable sense, 
to the mass-subcritical generalized Korteweg-de Vries (gKdV) equation. 
Especially, a use of the refined non-diagonal estimate removes several technical restrictions on
the previous work \cite{MS2} about the existence of the special 
non-scattering solution.  
\end{abstract}


\section{Introduction}

%
%

In this paper,
we consider space-time estimates for a solution of 
the Airy equation
\begin{equation}\label{A}
\left\{
\begin{aligned}
&\pt_t u+\pt_x^3 u=0,\quad
&& t,x\in\rre,\\
&u(0,x)=f(x),
&&  x\in\rre,
\end{aligned}
\right.
\end{equation}
where $f:\rre\to\rre$ is a given data. 
After the pioneering work by Strichartz \cite{St}, 
the space-time estimate for the dispersive 
equation has been studied by many authors in several directions
(see for instance \cite{Caz,KPV1} for the historical background of this topic).  
As for the Schr\"odinger equation, the Strichartz estimate 
for (\ref{A}) is well-known (see \cite{KPV1} for instance). 
Gr\"unrock \cite{G1} and the authors \cite{MS1} extended  
the Strichartz estimate 
for (\ref{A}) to the hat-Lebesgue space, more precisely 
we obtained the following estimate.

\begin{theorem}[generalized Strichartz' estimate \cite{G1,MS1}]
Let $(p,q)$ be a pair satisfying either $(p,q)=(\I,2),(4.\I)$ or
\[
	0\le\frac{1}{p}<\frac14,\quad\quad 0 \le \frac1q < \frac12 - \frac1p.
\]
Then, there exists a positive constant $C$ depends 
only on $\alpha$ and $s$ such that the inequality
\begin{equation}\label{eq:mixed}
	\tnorm{|\pt_x|^s e^{-t\d_x^3} f}_{L^p_x(\rre;L^q_t(\rre))}
	+
	\tnorm{|\pt_x|^{\frac1p} e^{-t\d_x^3} f}_{L^p_t(\rre;L^q_x(\rre))} 
	\le C\norm{f}_{\hat{L}^{\alpha}} 
\end{equation}
holds for any $f\in\hat{L}^{\alpha}$, where $\alpha$ and $s $ are given by 
\[
	\frac1{\alpha} = \frac2p + \frac1q,\quad s=-\frac1p+\frac2q
\]
and the space $\hat{L}^{\alpha}$ is defined for $1\le {\alpha} \le \infty$ by 
\begin{eqnarray*}
	\hat{L}^{\alpha}=\hat{L}^{\alpha}(\rre):=\{f\in{{\mathcal S}}'(\rre)|
	\norm{f}_{\hat{L}^{\alpha}} =\|\hat{f}\|_{L^{\alpha'}}<\infty\},
\end{eqnarray*}
and $\hat{f}$ stands for Fourier transform of $f$ in $x$,  
and $\alpha'$ denotes the H\"older conjugate of $\alpha$.  
\end{theorem}
The generalized Strichartz' estimate (\ref{eq:mixed}) is shown by interpolating 
the endpoint cases $(p,q)=(\I,2)$, $(4.\I)$, which corresponds 
to the well-known Kato's smoothing and Kenig-Ruiz estimates, 
and the diagonal case $p=q \in (4,\I]$. We refer the estimate 
in the diagonal case to as a Stein-Tomas estimate.

The aim of this paper is to obtain a refinement of 
the Strichartz/Stein-Tomas estimates for (\ref{A})
for data in a (generalized) hat-Morrey space, which 
is wider than the above hat-Lebesgue spaces (see Appendix A).
Let us first give its definition.

\begin{definition}[A Morrey and a hat-Morrey spaces]\label{Morrey}
For $j,k \in \Z$, let $\tau^{j}_{k}=[k2^{-j},(k+1)2^{-j})$ be a dyadic interval. 

\vskip1mm
\noindent
(i) For $1 \le \gamma \le  \beta \le \I$ and $\beta < \delta\le\infty$, we define a 
Morrey norm 
 by
\[
	\norm{f}_{M^\beta_{\gamma,\delta}} =
	\norm{|\tau^{j}_{k}|^{\frac1\beta-\frac1\gamma} \norm{f}_{L^{\gamma}(\tau^{j}_{k})} }_{\ell^{\delta}_{j,k}},
\]
where, the case $\beta=\gamma$ and $\gamma<\I$ is excluded.

\vskip1mm
\noindent
(ii) 
For $1\le \beta \le \gamma \le \I$ and $\beta' <  \delta\le \I$,
we define a hat-Morrey norm by
\[
	\norm{f}_{\hat{M}^\beta_{\gamma,\delta}} := \| \hat{f}\|_{M^{\beta'}_{\gamma',\delta}}=
	\norm{|\tau^{j}_{k}|^{\frac1\gamma-\frac1\beta} \tnorm{\hat{f}}_{L^{\gamma'}(\tau^{j}_{k})} }_{\ell^{\delta}_{j,k}} .
\]
Banach spaces $M^{\beta}_{\gamma,\delta}$ and 
$\hat{M}^{\beta}_{\gamma,\delta}$ are defined as 
sets of tempered distributions of which above norms are finite, respectively.
\end{definition}

One of the main motivation of this kind of improvement of the Strichartz 
estimates lies in its applications to nonlinear theory.
Especially, we are interested in construction of 
a special non-scattering solutions, which are minimal in a suitable sense,
to the mass-subcritical generalized Korteweg-de Vries (gKdV) equation:
\begin{equation}\tag{gKdV}\label{gKdV}
\left\{
\begin{aligned}
&\pt_t u+\pt_x^3 u=\mu\pt_x(|u|^{2\alpha}u),\qquad
&& t,x\in\rre,\\
&u(t_0,x)=u_{0}(x), 
&&  x\in\rre, 
\end{aligned}
\right.
\end{equation}
where $t_0 \in \R$, $u:\rre\times\rre\to\rre$ is an unknown function, 
$u_{0}:\rre\to\rre$ is a given data, and $0<\alpha<2$, $\mu\in\rre\backslash\{0\}$ 
are constants. 
We construct the minimal solution for (\ref{gKdV}) 
by using the concentration compactness argument by Kenig-Merle \cite{KM}. 
As explained in \cite[Section 1]{MS2}, a good well-posedness theory and
a decoupling (in)equality play a central role in the concentration compactness argument.
However, when $\alpha<2$, it seems difficult to derive those properties 
in the Sobolev or hat-Lebesgue spaces by several reasons.  
In \cite{MS2}, it turns out that a use of the generalized hat-Morrey space enables us 
to establish well-posedness theory good enough and to obtain
the concentration compactness 
lemma equipped with a decoupling inequality\footnote{Although the decoupling (in)equality is 
not obtained in $\hat{L}^\alpha$ with $\alpha\neq2$, 
it can be established in a wider space $\hat{M}^\alpha_{2,\delta}$ 
thanks to the local $L^2$ structure of $\hat{M}^\alpha_{2,\delta}$.}. 
Our estimate in Theorem \ref{thm:S} removes several technical restrictions made in \cite{MS2}.
See Subsection 1.2 below for more details. 

As far as the authors know, the refinement of the Stein-Tomas estimate in this direction first appeared in \cite{B1}
in a context of Schr\"odinger equation.
Besides its own interests, the refined estimate has been studied rather because of its application.
In \cite{B3}, Bourgain use the refined estimate to show
a concentration phenomenon of blow-up solutions for the two dimensional 
mass-critical nonlinear Schr\"odinger equation. 
After Bourgain, 
the refinement of Strichartz estimates are being used, for instance, in the estimate for the 
maximal function associated to the Schr\"odinger equation 
(Moyua, Vargas and Vega \cite{MVV1,MVV2}),
or in the linear profile decomposition in $L^2$-framework 
for the Schr\"odinger equation (Merle and Vega \cite{MV}, Carles and Keraani \cite{CK}, 
and B\'{e}gout and Vargas \cite{BV}).
As for the Airy equation (\ref{A}), 
Kenig, Ponce and Vega \cite{KPV3} showed the the refined estimate 
and applied it to a study of a concentration of 
blow-up solution for the mass-critical generalized KdV equation. 
By using the estimate, 
Shao \cite{Shao} proved the linear profile decomposition for Airy equation in $L^2$-framework.  

In all above studies, the refinements  
were restricted to the case $\alpha=2$ and the \emph{diagonal} case $p=q$. 
In \cite{MS2} the authors proved the refined Stein-Tomas estimate for (\ref{A}) 
in the case $\alpha\neq2$
and used it for proving existence of a minimal non-scattering solution for the 
\emph{mass-subcritical} generalized KdV equation in the $\hat{L}^\alpha$-framework.
A similar refinement in the Schr\"odinger case was done by the first author \cite{M3},
including its application to existence of a minimal non-scattering solution for 
the mass-subcritical nonlinear Schr\"odinger equation in $\hat{M}^\beta_{\gamma,\delta}$-framework. 
However, the refinement is still restricted to the diagonal case $p=q$.

Main purpose of this paper is to extend the refinement  
to the \emph{non-diagonal} case $p\neq q$, that is, we show refined Strichartz estimates in our terminology.
Our main theorems are as follows. We first give the 
estimate of the Airy equation in the space $L^{p}_x(\R; L^{q}_t(\R))$.  

\begin{theorem}[Refined Strichartz' estimate I]\label{thm:S} 
Let $\si\in(0,1/4)$. Let $(p,q)$ satisfy
\begin{eqnarray*}
0\le\frac1p\le\frac14-\si,\qquad\frac1q\le\frac12-\frac{1}{p}-\si.
\end{eqnarray*}
Define $\alpha$ and $s$ by 
\[
	\frac2p + \frac1q=\frac1\alpha,\qquad s= -\frac1p+\frac2q.
\]
Further, we define $\beta$, $\gamma$, and $\delta$ by
\begin{eqnarray*}
\frac1\beta=\frac1\alpha+\si, \qquad
\frac1\gamma=
\left\{
\begin{aligned}
&\frac1\beta - \frac1p\quad
&&\text{if}\ \ \frac1q\ge\frac1p+\si,\\
&\frac1\beta - \frac1q+\si
&&\text{if}\ \ \frac1q<\frac1p+\si,
\end{aligned}
\right.
\qquad \frac1\delta = \frac12 - \frac1{\max(p,q)}.
\end{eqnarray*}
Then, there exists a positive constant $C$ 
depending on $p,q,\sigma$ such that 
the inequality 
\begin{eqnarray}
\norm{|\d_x|^s e^{-t\d_x^3} f}_{L^{p}_x(\R; L^{q}_t(\R))}
\le C\||\d_x|^{\si} f\|_{\hat{M}^{\beta}_{\gamma,\delta}}
\label{qwe}
\end{eqnarray}
holds for any $f\in |\d_x|^{-\si}\hat{M}^{\beta}_{\gamma,\delta}$.
\end{theorem}

\begin{remark} For the diagonal case $p=q$, the inequality (\ref{qwe}) 
holds for $\si=0$, see \cite[Theorem B.1]{MS2}. 
\end{remark}

Next we give the estimate for the Airy equation in 
the space $L^{p}_t(\R; L^{q}_x(\R))$.

\begin{theorem}[Refined Strichartz' estimate II]\label{thm:T} 
Let $\si\in(0,1/4)$. Let $(p,q)$ satisfy
\begin{eqnarray*}
0\le\frac1p\le\frac14,\qquad\frac1q\le\frac12-\frac{1}{p}-\si.
\end{eqnarray*}
Define $\alpha$ by 
\[
	\frac2p + \frac1q=\frac1\alpha.
\]
Further, we define $\gamma$ and $\delta$ by
\begin{eqnarray*}
\frac1\gamma
=
\left\{
\begin{aligned}
&\frac1\alpha - \frac1p+\sigma\quad
&&\text{if}\ \ \frac1q\ge\frac1p-\si,\\
&\frac1\alpha - \frac1q
&&\text{if}\ \ \frac1q<\frac1p-\si,
\end{aligned}
\right.
\qquad \frac1\delta = \frac12 - \frac1{\max(p,q)}.
\end{eqnarray*}
Then, there exists a positive constant $C$ 
depending on $p,q$ such that 
the inequality 
\begin{eqnarray}
\norm{|\d_x|^{\frac1p} e^{-t\d_x^3} f}_{L^{p}_t(\R; L^{q}_x(\R))}
\le C\|f\|_{\hat{M}^{\alpha}_{\gamma,\delta}}
\label{qwe2}
\end{eqnarray}
holds for any $f\in \hat{M}^{\alpha}_{\gamma,\delta}$.
\end{theorem}

We briefly outline the proofs for Theorems \ref{thm:S}, the proof of Theorem \ref{thm:T} is similar. 
The diagonal case $p=q$ can be handled by the bilinear technique 
as in \cite{Shao,MS2}. 
However, this approach does not work well in the non-diagonal case. 
Furthermore, due to lack of an interpolation between the Morrey space and the Lebesgue space,
the desired estimate does not follow by a simple interpolation. 
To overcome those difficulties, we take another approach 
which is based on \cite{BV,KR,MS2,VV}. 
As in the diagonal case, we first rewrite the square of 
the left hand side of (\ref{qwe}) into a bi-linear oscillatory integral. 
We then split the domain of spacetime integral into 
infinitely many rectangles by a Whitney type decomposition. 
By the decomposition, the bilinear form is rewritten as the infinite sum of the bilinear forms of which Fourier supports are
compact and do not intersect each other.
To justify the above decomposition in $L^p_x L^q_t$ or $L^p_t L^q_x$,
we have to add a small \emph{margin} to each rectangles in the Whitney decomposition in purpose of smooth cutoff.
Obviously, this margin produces many doublings which disturb orthogonality of the forms.
However, if the margin is putted so nicely that the resulting doubling is acceptable
then we obtain the desired estimate.
The property is 
summarized as an \emph{almost orthogonal property} of the Fourier supports of the forms.

In the Schr\"odinger case, we can put such margin so that the almost orthogonal property is valid (see \cite{BV}).
However, in the Airy case, the cubic dispersion makes the situation much worse and it seems there is no way to put
margin necessary for smooth cutoff.
An idea here is to put the margin \emph{only in time direction}.
Although this requires an unpleasant restriction $\sigma>0$ 
in Theorems \ref{thm:S} and \ref{thm:T}, we recover the almost orthogonal property.
See Proposition \ref{orthgnal} and Remark \ref{rem:si} for the detail.

Next we give several applications of our refinement estimates. 

\subsection{Application 1 -- well-posedness for generalized KdV equation}

As the first application of the refinement of Strichartz' estimates, 
we show the well-posedness of (\ref{gKdV}) in the scale critical 
$\hat{M}^{\beta}_{\gamma,\delta}$ space. 

Local and global well-posedness of the Cauchy problem (\ref{gKdV}) in a scale critical or subcritical Sobolev space 
$H^{s}(\rre)$, $s\ge s_{\alpha}$ has been studied by many authors, 
where $s_\alpha$ is a scale critical exponent, i.e., $s_\alpha:=1/2-1/\alpha$.   
A fundamental work on local well-posedness is due 
to Kenig, Ponce and Vega \cite{KPV2}. They proved that (\ref{gKdV}) is locally 
well-posed in $H^{s}(\rre)$ with $s>3/4$ ($\alpha=1/2$), 
$s\ge1/4$ ($\alpha=1$), $s\ge1/12$ ($\alpha=3/2$) and 
$s\ge s_{\alpha}$ ($\alpha\ge2$). 
Furthermore, in \cite{KPV2} Kenig, Ponce and Vega 
proved the small data global well-posedness 
and scattering of (\ref{gKdV}) 
in the scale critical space $\dot{H}^{s_{\alpha}}$ 
for $\alpha\ge2$.  
Tao \cite{T} proved 
global well-posedness for small data for (\ref{gKdV}) 
with the quartic nonlinearity $\mu\pt_{x}(u^{4})$ 
in $\dot{H}^{s_{3/2}}$, see also Koch and Marzuola \cite{KoM} for 
the simplified the proof of \cite{T} and made an extension. 
Recently, the authors \cite{MS1} obtained 
global well-posedness for small data for (\ref{gKdV}) 
in the scale critical space $\hat{L}^\alpha$ with $8/5 <\alpha <10/3$. 

We consider global well-posedness for small data for (\ref{gKdV}) 
in a scale critical hat-Morrey space $|\pt_{x}|^{-\si}
\hat{M}^{\beta}_{\gamma,\delta}$ space. 
It is known that the nonlinear Schr\"{o}dinger equation 
is globally well-posed for small data in the scale critical 
$\hat{M}^{\beta}_{\gamma,\delta}$ space, see \cite{M3,MS2}. 
In the Schr\"odinger case, we only need 
the \emph{diagonal} refined estimate to obtain well-posedness. 
On the other hand, as for (\ref{gKdV}), due to the presence of derivatives in the 
nonlinearity, we also need the \emph{non-diagonal} refined Strichartz estimate for (\ref{A}) 
to yield a similar well-posedeness result.
In this paper, 
by using the refined 
 \emph{non-diagonal} estimate in Theorem \ref{thm:S}, we 
shall prove the global well-posedness for small data for (\ref{gKdV}) in the scale critical 
$\hat{M}^{\beta}_{\gamma,\delta}$ space. 

\begin{assumption}\label{A:lwp}
Let $5/3<\alpha \le 20/9$ and $0<\sigma \le \min( 3/5-1/\alpha,1/4-2/(5\alpha))$.
Define $\beta$ by $1/\beta=1/\alpha+\sigma$.
Let $\gamma $ and $\delta$ satisfy 
\[
	\frac4{5\alpha} + 2\sigma \le \frac1\gamma < \frac1\beta,\quad
	\frac12-\frac1{5\alpha} \le \frac1\delta < \frac1{\beta'}.
\]
\end{assumption}

\begin{theorem}[Local well-posedness in
$|\pt_{x}|^{-\sigma}\hat{M}^{\beta}_{\gamma,\delta}$]\label{thm:lwp}
Suppose $\alpha$, $\sigma$, $\beta$, $\gamma$. and $\delta$ satisfy Assumption \ref{A:lwp}.
Then, the initial value problem \eqref{gKdV} is locally well-posed in
$|\pt_{x}|^{-\sigma}\hat{M}^{\beta}_{\gamma,\delta}$. More precisely,
for any $|\pt_{x}|^{\sigma}u_{0}\in \hat{M}^{\beta}_{\gamma,\delta}(\rre)$,
there exist an interval $I=I(u_0)$ and a unique solution 
to \eqref{gKdV} satisfying
\begin{eqnarray}
\qquad u\in C(I;|\pt_{x}|^{-\sigma}\hat{M}^{\beta}_{\gamma,\delta}(\rre)) 
\cap L^{\frac{5\alpha}{2}}_x (\R;L^{5\alpha}_t(I)) \cap |\pt_x|^{-\frac1{3\beta}-\sigma} 
L^{3\beta}_{t,x}(I \times \R).
\label{sol}
\end{eqnarray}
For any compact subinterval $I'\subset I$, 
there exists a neighborhood $V$ of $u_{0}$ in 
$|\pt_{x}|^{-\sigma}\hat{M}^{\beta}_{\gamma,\delta}(\rre)$ such that 
the map $u_{0}\mapsto u$ from $V$ into the class defined by \eqref{sol}
with $I'$ instead of $I$ is Lipschitz continuous. 
The solution satisfies $u(t) - e^{-(t-t_0)\d_x^3}u(t_0) \in C(I; \hat{L}^\alpha \cap |\d_x|^{-\sigma} \hat{L}^\beta)$
for any $t_0 \in I$.
\end{theorem}

Throughout this paper, we call a function $u$ which satisfies \eqref{sol} and solves the corresponding integral equation
as a $|\d_x|^{-\sigma}\hat{M}^\beta_{\gamma,\delta}$-solution to \eqref{gKdV} on an interval $I$.

\begin{theorem}[Small data scattering in $|\pt_{x}|^{-\sigma}\hat{M}^{\beta}_{\gamma,\delta}$]\label{thm:SDS} 

Suppose $\alpha$, $\sigma$, $\beta$, $\gamma$. and $\delta$ satisfy Assumption \ref{A:lwp}.
Then, there exists $\varepsilon_{0}>0$ 
such that if $|\pt_{x}|^{\sigma}u_{0}\in \hat{M}^{\beta}_{\gamma,\delta}(\rre)$ 
satisfies $\||\pt_{x}|^{\sigma}u_{0}\|_{\hat{M}^{\beta}_{\gamma,\delta}}\le\varepsilon_{0}$,
then the solution $u(t)$ to \eqref{gKdV} given in Theorem \ref{thm:lwp}
is global in time and scatters for both time directions.
Moreover,
\begin{equation}\label{SDSbound}
\||\d_x|^{\si}u\|_{L_{t}^{\infty}(\R; \hat{M}_{\gamma,\delta}^{\beta})}
+\|u\|_{L^{\frac{5\alpha}{2}}_x (\R;L^{5\alpha}_t(\R))}
\le 2 \||\d_x|^{\si}u_{0}\|_{\hat{M}_{\gamma,\delta}^{\beta}}.
\end{equation} 
\end{theorem}

\begin{remark} 
For scale \emph{subcritical spaces}, there are many results on 
the small data scattering for the generalized KdV equation 
(\ref{gKdV}) for $\alpha\ge1$, see \cite{CW,DZ,HN1} for instance.
\end{remark}

\subsection{Application 2 -- existence of minimal non-scattering solution}

We next apply the refined Strichartz estimate to 
construct a minimal non-scattering solution to 
mass-subcritical generalized KdV equation (\ref{gKdV}). 
As for \eqref{gKdV}, the mass-critical 
case $\alpha=2$ is most extensively studied in this direction. 
Killip-Kwon-Shao-Visan \cite{KKSV} constructed 
a minimal blow-up solution to the mass critical KdV equation with the focusing nonlinearity 
in the framework of $L^{2}$. 
Dodson \cite{D2} proved the global well-posedness and scattering in $L^{2}$ 
for the mass critical KdV equation with the defocusing nonlinearity. 

The authors \cite{MS2} showed a existence of a minimal non-scattering solution of 
\eqref{gKdV} with the mass-subcritical case. 
We constructed the critical element by establishing 
the concentration compactness in the framework of 
$\hat{M}^\alpha_{2,\delta}$ space. 
On the other hand, well-posedness result was not proved in $\hat{M}^\alpha_{2,\delta}$ but
in $\hat{L}^\alpha$ 
due to lack of non-diagonal refined Strichartz estimate.
This disagreement caused some technical restrictions in the previous result \cite{MS2}.
In this paper, by using the non-diagonal refined Strichartz estimate (Theorem \ref{thm:S}), 
we resolve the disagreement and show existence of critical element under a reasonable assumption.

Before we state our main theorems in this subsection, 
we introduce several notation. In the rest of this section, 
a solution always implies a $|\d_x|^{-\sigma}\hat{M}^\beta_{2,\delta}$-solution
unless otherwise stated. We introduce a deformations associated with the 
function space $|\d_{x}|^{-\sigma}\hat{M}^{\beta}_{2,\delta}$:
\begin{itemize}
\item Translation in Physical side: $(T(y)f)(x) := f(x-y)$.
\item Airy flow: $(A(t) f)(x) = (e^{-t\pt_x^3} f)(x)$.
\item Dilation (scaling): $(D(h) f)(x) = h^{\alpha} f(hx)$.
\end{itemize}
Note that $|\d_{x}|^{-\sigma}\hat{M}^{\beta}_{2,\delta}$-norm is invariant under the 
above group actions. 

For a solution $u$ on $I$, take $t_0 \in I$ and set 
\begin{align*}
T_{\mathrm{max}}&:=
\sup\left\{T>t_0 \ |\
u(t) \text{ can be extended to a solution on }  [t_0,T).
\right\},\\
T_{\mathrm{min}}&:=
\sup\left\{T>-t_0\ |\
u(t) \text{ can be extended to a solution on }(-T,t_0].
\right\},\\
I_{\max}&=I_{\max}(u):=(-T_{\mathrm{min}},T_{\mathrm{max}}).
\end{align*}
\begin{definition}[Scattering]\label{def:scattering}
We say a solution $u(t)$ scatters forward in time
(resp.\ backward in time)
if $T_{\mathrm{min}}=\infty$ (resp. $T_{\mathrm{max}}=\I$) and if $|\d_x|^\sigma e^{t\d_x^3} u(t)$ converges in $\hat{M}^{\beta}_{2,\delta}$ as $t\to\I$ (resp. $t\to-\I$).
\end{definition}
We define 
\[
	E_{1}:=
 	\inf\left\{ \inf_{t\in I_{\max}}\tnorm{|\d_x|^\sigma u(t)}_{\hat{M}^\beta_{2,\delta}} \left|
 	\begin{aligned}
	&u(t)\text{ is a solution to }\eqref{gKdV} \text{ that}\\
	&\text{does not scatter forward in time.}
	\end{aligned}
	\right.\right\}.
\]
Theorem \ref{thm:SDS} is represented as $E_1>0$.
Remark that it holds that
\[
	E_{1}=
 	\inf\left\{ \tnorm{|\d_x|^\sigma u(0)}_{\hat{M}^\beta_{2,\delta}} \left|
 	\begin{aligned}
	&u(t)\text{ is a solution to }\eqref{gKdV}\text{ that does}\\
	&\text{not scatter forward in time, }0\in I_{\max}(u).
	\end{aligned}
	\right.\right\}.
\]
by the time translation symmetry.
Further, one sees that $E_1$ is the supremum of the number $\eps_0$ for which Theorem
\ref{thm:SDS} is true.

We also introduce another infimum value.
\begin{eqnarray*}
	E_{2}&:=&
 	\inf\left\{ \varlimsup_{t \uparrow T_{\max}}\tnorm{|\d_x|^\sigma u(t)}_{\hat{M}^\beta_{2,\delta}} \left|
 	\begin{aligned}
	&u(t)\text{ is a solution to }\eqref{gKdV} \text{ that}\\
	&\text{does not scatter forward in time.}
	\end{aligned}
	\right.\right\}.
\end{eqnarray*}
By definition, $E_1 \le E_2 \le \norm{|\d_x|^\sigma Q}_{\hat{M}^\beta_{2,\delta}}$.
For another characterization of this quantity, see Remark \ref{rmk:E2Ec}.
The goal is to determine the explicit value of $E_j$ ($j=1,2$).
Here, we will show that existence of minimizers to both $E_1$ and $E_2$,
which would be a important step.

In what follows, we consider the focusing case $\mu=-1$ only.
However, the focusing assumption is used only for assuring $E_j$ are finite.
Our analysis work also in the defocusing case $\mu=+1$ if we assume $E_j$ are finite.

\vskip2mm

\begin{assumption}\label{A:min}
We suppose Assumption \ref{A:lwp} with $\gamma=2$ and exclude the endpoint cases, i.e.,
Let $5/3 < \alpha < 12/5$ and $\max(0,1/2-1/\alpha) <\sigma < \min(3/5-1/\alpha,
1/4-2/(5\alpha))$.
Define $\beta \in (5/3,2)$ by $1/\beta = 1/\alpha+\sigma$
and let $1/\delta \in (1/2-1/(5\alpha),1/\beta')$.
\end{assumption}

\begin{theorem}[Analysis of $E_1$]\label{thm:minimal}
Suppose that Assumption \ref{A:min} is satisfied.
Then, $0<E_1 \le c_\alpha \norm{|\d_x|^\sigma Q}_{\hat{M}^\beta_{2,\delta}}$.
Furthermore, there exists a minimizer $u_1(t)$ to $E_1$ in the following sense:
$u_1(t)$ is a solution to \eqref{gKdV} with
maximal interval $I_{\mathrm{max}}(u_1) \ni 0$ and
\begin{enumerate}
\item $u_1(t)$ does not scatter forward in time;
\item $u_1(t)$ attains $E_1$ in such a sense that either one of the following two properties holds;
\begin{enumerate}
\item $\norm{|\d_x|^\sigma u_1(0)}_{\hat{M}^\beta_{2,\delta}} = E_1$;
\item $u_1(t)$ scatters backward in time and $u_{1,-}:= \lim_{t\to-\I} e^{t\d_x^3}u_1(t)$ satisfies
$\norm{|\d_x|^\sigma u_{1,-}}_{\hat{M}^\beta_{2,\delta}} = E_1$.
\end{enumerate}
\end{enumerate}
\end{theorem}

\begin{remark}
Let us mention the difference between the previous results in \cite{KKSV,MS2}.
In these papers, a priori knowledge of the relation between value of $E_1$ and the same value for a corresponding nonlinear Schr\"odinger equation is assumed.
In our theorem, we do not need this kind of assumption.
The assumption is used to exclude the case where $E_1$ is attained by a sequence of initial data of the form
$f(x) \cos (\xi_n x)$ with $\xi_n\to\I$ as $n\to\I$.
The case may happen because the state spaces used in \cite{KKSV,MS2} are $\hat{L}^\alpha$, in which the operation $e^{ix\xi}$ is unitary.
In our case, the state space $|\d_x|^{-\sigma} \hat{M}^\beta_{2,\delta}$ contains derivative and so 
the above case does not take place.
Thus, we do not need the assumption.
\end{remark}

\begin{theorem}[Analysis of $E_2$]\label{thm:minimal2}
Suppose that Assumption \ref{A:min} is satisfied.
Then, $E_1\le E_2 \le \norm{|\d_x|^\sigma Q}_{\hat{M}^\beta_{2,\delta}}$.
Furthermore, there exists a minimizer $u_2(t)$ to $E_2$ in the following sense:
$u_2(t)$ is a solution to \eqref{gKdV} with
maximal interval $I_{\mathrm{max}}(u_2) \ni 0$ and
\begin{enumerate}
\item $u_2(t)$ does not scatter forward and backward in time;
\item Three quantities 
$$\displaystyle \sup_{t \in \R} \norm{|\d_x|^\sigma u_2(t)}_{\hat{M}^\beta_{2,\delta}},\quad \varlimsup_{t\uparrow T_{\max}} \norm{|\d_x|^\sigma u_2(t)}_{\hat{M}^\beta_{2,\delta}},\quad \varlimsup_{t\downarrow T_{\min}} \norm{|\d_x|^\sigma u_2(t)}_{\hat{M}^\beta_{2,\delta}}$$ are equal to $E_2$.
\item $u_2(t)$ is precompact modulo symmetries, i.e., there exist a scale function $N(t): I_{\max} \to \R_+$ and
a space center $y(t): I_{\max}\to \R$ such that the set
$\{ (D(N(t))T(y(t)))^{-1} u_2(t)\ |\ t \in I_{\max} \} \subset 
	|\d_x|^{-\sigma} \hat{M}^\beta_{2,\delta}$
is precompact.
\end{enumerate}
\end{theorem}

\begin{remark}\label{rmk:E2Ec}
We give another characterization of $E_2$.
For $E\ge0$, we define
\[
	\mathcal{L}(E):= \sup\left\{ \norm{u}_{L^{\frac{5\alpha}2}_x(\R;L^{5\alpha}_t(I))}\left|
	\begin{aligned}
	&u(t)\in C(I;|\d_x|^{-\sigma}\hat{M}^\beta_{2,\delta}) \text{ is a solution}\\
	&\text{to \eqref{gKdV} on a compact interval }I \\
	&\text{such that } \max_{t\in I}\norm{|\d_x|^\sigma u(t)}_{\hat{M}^\beta_{2,\delta}} \le E
	\end{aligned}
	\right.\right\}.
\]
Remark that $\mathcal{L}:[0,\I)\to[0,\I]$ is non-decreasing.
Then, it holds that $E_2= \sup\{E\ |\ \mathcal{L}(E)<\I\}= \inf \{E\ |\ \mathcal{L}(E)=\I\}$.
\end{remark}

The following notation will be used throughout this 
paper: We use the notation 
$A\sim B$ to represent $C_{1}A\le B\le
C_{2} A$ for some constants $C_{1}$ and $C_{2}$.
We also use the notation $A\lesssim B$ to denote $A\le CB$ for
some constant $C$. The operator $|\pt_x|^s=(-\pt_x^2)^{s/2}$ denotes 
the Riesz potential of order $-s$. 
For $1\le p,q\le\infty$ and $I\subset\rre$, 
let us define a space-time norm
\begin{eqnarray*}
\|f\|_{L_t^qL_x^p(I)}&=&
\|\|f(t,\cdot)\|_{L_x^p(\rre)}\|_{L_t^q(I)},\\
\|f\|_{L_x^pL_t^q(I)}&=&
\|\|f(\cdot,x)\|_{L_t^q(I)}\|_{L_x^p(\rre)}.
\end{eqnarray*}

The rest of the article is organized as follows. 
In Section 2, we prove Theorems \ref{thm:S} and \ref{thm:T}. 
In Section 3, we shall show the well-posedness and the small data scattering 
for \eqref{gKdV} (Theorems \ref{thm:lwp} and \ref{thm:SDS}) by using 
the refined Strichartz estimate obtained by Theorem \ref{thm:S} 
and the contraction mapping principle. 
Finally in Section 4, we construct a minimal non-scattering solution 
to \eqref{gKdV} (Theorems \ref{thm:minimal} and \ref{thm:minimal2}) 
by using the concentration compactness. In Appendix, we summarize 
the embedding properties of the generalized Morrey space.

%
%

\section{Proof of Strichartz estimates in the hat-Morrey space}

%
%

In this section we derive the refinement 
version of the Strichartz estimates for solution to 
(\ref{A}) (Theorems \ref{thm:S} and \ref{thm:T}) 
by using the argument used in \cite{BV,KR,MS2,VV}. 
We show the inequality (\ref{qwe}) only, the proof of (\ref{qwe2}) being similar.

\subsection{Whitney decomposition} 

To show the inequality (\ref{qwe}), we first reduce the linear form into a \emph{bilinear} form:
\begin{eqnarray*}
\||\pt_{x}|^{s}e^{-t\pt_{x}^{3}}f\|_{L_{x}^{p}L_{t}^{q}}^{2}
=\|||\pt_{x}|^{s}e^{-t\pt_{x}^{3}}f|^{2}\|_{L_{x}^{\frac{p}{2}}L_{t}^{\frac{q}{2}}}.
\end{eqnarray*}
If $f$ is a real valued function, then 
\begin{eqnarray*}
(e^{-t\pt_{x}^{3}}f)(x)
=\sqrt{\frac{2}{\pi}}
Re\left[\int_{0}^{\infty}
e^{ix\xi+it\xi^{3}}\hat{f}(\xi)d\xi\right].
\end{eqnarray*}
Hence
\begin{eqnarray*}
\lefteqn{||\pt_{x}|^{s}e^{-t\pt_{x}^{3}}f|^{2}}\\
&=&
\frac{1}{\pi}Re
\int_{0}^{\infty}\int_{0}^{\infty}
e^{ix(\xi+\eta)+it(\xi^{3}+\eta^{3})}
|\xi\eta|^{s}
\hat{f}(\xi)\hat{f}(\eta)d\xi d\eta\\
& &
+\frac{1}{\pi}Re
\int_{0}^{\infty}\int_{0}^{\infty}
e^{ix(\xi-\eta)+it(\xi^{3}-\eta^{3})}|\xi\eta|^{s}
\hat{f}(\xi)\overline{\hat{f}(\eta)}d\xi d\eta.
\end{eqnarray*}
We now introduce a Whitney decomposition. 
Let ${{\mathcal D}}_{+}=\{[k2^{-j},(k+1)2^{-j})|j\in{{\mathbb Z}}, 
0\le k\in{{\mathbb Z}}\}$. 
For $\tau_{k}^{j}$, $\tau_{\ell}^{j}
\in{{\mathcal D}}_{+}$, we define a binary relation
\begin{eqnarray}\label{def:binary}
\tau_{k}^{j}\sim\tau_{\ell}^{j}
\ \Leftrightarrow\ 
\begin{cases}
\ell-k=-2,2,3\quad\ \ \text{if}\ k\ \text{is\ even},\\
\ell-k=-3,-2,2\quad\text{if}\ k\ \text{is\ odd}.
\end{cases}
\end{eqnarray}
Then, we have 
$(\rre_{+}\times\rre_{+})\backslash\{(\xi,\xi)|\xi\ge0\}=
\bigcup\{\tau_{k}^{j}\times\tau_{k}^{\ell}|\tau_{k}^{j}\in{{\mathcal D}}_{+}, \tau_{\ell}^{j}:\tau_{\ell}^{j}\sim\tau_{k}^{i}\}$.
The Whitney decomposition gives us 
\begin{eqnarray}
\label{mob}\\
\lefteqn{||\pt_{x}|^{s}e^{-t\pt_{x}^{3}}f|^{2}}\nonumber\\
&=&
\frac{1}{\pi}
\sum_{\tau_{k}^{j}\in{{\mathcal D}}_{+}}
\sum_{\tau_{\ell}^{j}:\tau_{\ell}^{j}\sim\tau_{k}^{i}}
Re
\int_{\tau_{k}^{j}}\int_{\tau_{\ell}^{j}}
e^{ix(\xi+\eta)+it(\xi^{3}+\eta^{3})}
|\xi\eta|^{s}
\hat{f}(\xi)\hat{f}(\eta)d\xi d\eta\nonumber\\
& &+
\frac{1}{\pi}\sum_{\tau_{k}^{j}\in{{\mathcal D}}_{+}}
\sum_{\tau_{\ell}^{j}:\tau_{\ell}^{j}\sim\tau_{k}^{j}}
Re
\int_{\tau_{k}^{j}}\int_{\tau_{\ell}^{j}}
e^{ix(\xi-\eta)+it(\xi^{3}-\eta^{3})}
|\xi\eta|^{s}
\hat{f}(\xi)\overline{\hat{f}(\eta)}d\xi d\eta\nonumber\\
&=&
2Re\sum_{\tau_{k}^{j}\in{{\mathcal D}}_{+}}
\sum_{\tau_{\ell}^{j}:\tau_{\ell}^{j}\sim\tau_{k}^{j}}
|\pt_{x}|^{s}e^{-t\pt_{x}^{3}}f_{\tau_{k}^{j}}
|\pt_{x}|^{s}e^{-t\pt_{x}^{3}}f_{\tau_{\ell}^{j}}
\nonumber\\
& &+2Re\sum_{\tau_{k}^{j}\in{{\mathcal D}}_{+}}
\sum_{\tau_{\ell}^{j}:\tau_{\ell}^{j}\sim\tau_{k}^{j}}
|\pt_{x}|^{s}e^{-t\pt_{x}^{3}}f_{\tau_{k}^{j}}
\overline{|\pt_{x}|^{s}e^{-t\pt_{x}^{3}}f_{\tau_{\ell}^{j}}}
\nonumber\\
&=:&I_{1}+I_{2},\nonumber
\end{eqnarray}
where $\hat{f}_{I}(\xi)={{\bf 1}}_{I}(\xi)\hat{f}(\xi)$. 
A simple calculation leads 
\begin{eqnarray*}
\lefteqn{\supp{{\mathcal F}}_{t,x}
[|\pt_x|^se^{-t\pt_{x}^{3}}f_{\tau_{k}^{j}}
|\pt_x|^se^{-t\pt_{x}^{3}}f_{\tau_{\ell}^{j}}](\tau,\xi)}\\
&\subset&\{(\xi_{1}^{3}+\xi_{2}^{3},\xi_{1}+\xi_{2})|\xi_{1}\in \tau_{k}^{j},\xi_{2}\in 
\tau_{\ell}^{j}\}\\
&\subset&A_{j,k,\ell},
\end{eqnarray*}
where $A_{j,k,\ell}$ is given by 
\begin{eqnarray}
A_{j,k,\ell}=\left\{(\tau,\xi)\left|\ \frac{k+\ell}{2^{j}}\le\xi\le
\frac{k+\ell+2}{2^{j}},
\right.\tau\ \text{satisfies}\ (\ref{cc1})\right\}
\label{dd1}
\end{eqnarray}
with
\begin{eqnarray}
\ \ \ \left\{
\begin{aligned}
&\frac34\frac{(k-\ell-1)^{2}}{2^{2j}}\xi
\le \tau-\frac14\xi^{3}\le\frac34\frac{(k-\ell+1)^{2}}{2^{2j}}\xi
\\
&\qquad\qquad\qquad\qquad\qquad\qquad\qquad\qquad\text{if}\ \ell-k=-3,-2,\\
&\frac34\frac{(k-\ell+1)^{2}}{2^{2j}}\xi
\le \tau-\frac14\xi^{3}\le\frac34\frac{(k-\ell-1)^{2}}{2^{2j}}\xi\\
&\qquad\qquad\qquad\qquad\qquad\qquad\qquad\qquad\text{if}\ \ell-k=2,3.
\end{aligned}
\right.\label{cc1}
\end{eqnarray}

In a similar way, we see 
\begin{eqnarray*}
\lefteqn{\supp{{\mathcal F}}_{t,x}
[|\pt_x|^se^{-t\pt_{x}^{3}}f_{\tau_{k}^{j}}
\overline{|\pt_x|^se^{-t\pt_{x}^{3}}f_{\tau_{\ell}^{j}}}](\tau,\xi)}\\
&\subset&\{(\xi_{1}^{3}+\xi_{2}^{3},\xi_{1}+\xi_{2})|\xi_{1}\in \tau_{k}^{j},\xi_{2}\in 
\tau_{-\ell-1}^{j}\}\\
&\subset&B_{j,k,\ell},
\end{eqnarray*}
where $B_{j,k,\ell}$ is given by 
\begin{eqnarray}
\qquad B_{j,k,\ell}=\left\{(\tau,\xi)\ \left|\ \frac{k-\ell-1}{2^{j}}\le\xi\le
\frac{k-\ell+1}{2^{j}},
\right.\tau\ \text{satisfies}\ (\ref{cc2})\right\},
\label{dd2}
\end{eqnarray}
with
\begin{eqnarray}
\ \ \left\{
\begin{aligned}
&\frac34\frac{(k+\ell)^{2}}{2^{2j}}\xi
\le \tau-\frac14\xi^{3}\le\frac34\frac{(k+\ell+2)^{2}}{2^{2j}}\xi
\ \ \ \text{if}\ \ell-k=-3,-2,\\
&\frac34\frac{(k+\ell+2)^{2}}{2^{2j}}\xi
\le \tau-\frac14\xi^{3}\le\frac34\frac{(k+\ell)^{2}}{2^{2j}}\xi
\ \ \ \text{if}\ \ell-k=2,3.
\end{aligned}
\right.\label{cc2}
\end{eqnarray}

\subsection{Key estimates}
Let us introduce two preliminary estimates
associated with the set $A_{j,k,\ell}$ and $B_{j,k,\ell}$ given in the previous section.

For a closed domain $R\subset\rre^{2}$ and $\lambda>0$, we define
\begin{eqnarray*}
R_{+\lambda}=\{(\tau+\tau',\xi)|(\tau,\xi)\in R,\ -\lambda\le\tau'
\le\lambda\}.
\end{eqnarray*}
The set $R_{+\lambda}$ is an enlargement of $R$ in $\tau$-direction.
Let $\varphi\in C_0^\I(\R)$ be a nonnegative function such that
$\supp\varphi\subset[-1,1]$ and $\int_{-1}^{1}\varphi(x) dx=1$.
Define a cut-off function
\[
	\psi_{R,{\lambda}}(\tau,\xi) := \left[
\frac2{\lambda}\varphi\left(\frac{2 }{\lambda} (\cdot)  \right)
\ast_{\tau}{\bf 1}_{R_{+\frac{\lambda}2}}(\cdot,\xi)\right](\tau),
\]
where ${\bf 1}_{\Omega}(\tau,\xi)$ 
is a characteristic function supported on $\Omega\subset\rre^{2}$. 
Note that $\psi_{R,{\lambda}}$ is smooth function with respect to 
$\tau$ variable. Furthermore, $\psi_{R,{\lambda}}$ 
satisfies $0 \le \psi_{R,{\lambda}}\le 1$,
$\psi_{R,{\lambda}} \equiv 1$ on $R$, and $\supp \psi_{R,{\lambda}} \subset R_{+{\lambda}}$.
We define a Fourier multiplier $P_{R,{\lambda}}$ by
\begin{eqnarray}
	(P_{R,{\lambda}}f)(t,x) &:=& \mathcal{F}^{-1}_{\tau,\xi}[\psi_{R,{\lambda}} 
	\mathcal{F}_{t,x} f](t,x)\label{def:PR}\\
	&=&\(\F^{-1}_{\tau}[\varphi]\(\frac{\lambda}{2}t\)\F^{-1}_{\tau,\xi}[{\bf 1}_{R_{+\frac{\lambda}{2}}}]
	 \ast f\)(t,x)\nonumber.
\end{eqnarray}
Let $\Lambda=\{(j,k,\ell)\in{{\mathbb Z}}\times{{\mathbb Z}}_{\ge0}
\times{{\mathbb Z}}_{\ge0}|\ell-k=-3,-2,2,3\}$. 
For $(j,k,\ell)\in\Lambda$, 
we let two families of sets $\{A_{j,k,\ell}\}$ and 
$\{B_{j,k,\ell}\}$ be as in (\ref{dd1}) and (\ref{dd2}), respectively.
We further introduce
\begin{equation}\label{tAB}
	\widetilde{A}_{j,k,\ell} = (A_{j,k,\ell})_{+\frac{k}{100\times2^{3j}}},\quad
	\widetilde{B}_{j,k,\ell} = (B_{j,k,\ell})_{+\frac{k}{100\times2^{3j}}}.
\end{equation}
As we explained in Introduction, 
the following finite doubling properties of the two families 
$\{\tilde{A}_{j,k,\ell}\}$ and 
$\{\tilde{B}_{j,k,\ell}\}$ play an important role 
in the proof of  Theorems \ref{thm:S} and \ref{thm:T}.

\begin{proposition}[Almost orthogonality]\label{orthgnal}
Let $X=A$ or $B$. Then the inequality 
\begin{eqnarray}
	\sum_{(j,k,\ell)\in\Lambda} 
	{\bf 1}_{\mathop{\widetilde{X}_{j,k,\ell}}} (\tau,\xi) \le 12
	\label{orthg}
\end{eqnarray}
holds for almost all $(\tau,\xi)\in\rre^{2}$, 
where $\Lambda=\{(j,k,\ell)\in{{\mathbb Z}}\times{{\mathbb Z}}_{\ge0}
\times{{\mathbb Z}}_{\ge0}|\ell-k=-3,-2,2,3\}$. 
\end{proposition}

\begin{proof}[Proof of Proposition \ref{orthgnal}]
We first note that 
$\Lambda=\{(j,k,k+m)\in{{\mathbb Z}}\times{{\mathbb Z}}_{\ge0}
\times{{\mathbb Z}}_{\ge0}|m=-3,-2,2,3\}$. 
By (\ref{dd1}), we see 
\begin{eqnarray*}
\widetilde{A}_{j,k,k+m}
\subset
\left\{(\tau,\xi)\ \left|\ \frac{2k+m}{2^{j}}\le\xi\le
\frac{2(k+1)+m}{2^{j}},
\right.\tau\ \text{satisfies}\ (\ref{cc11})\right\},
\end{eqnarray*}
where 
\begin{eqnarray}
\qquad
\left\{
\begin{aligned}
&\frac{3}{16}\frac{(m+1)^{2}}{2^{2j}}\xi
\le \tau-\frac14\xi^{3}\le3\frac{(m-1)^{2}}{2^{2j}}\xi\quad
&&\text{if}\ m=-3,-2,\\
&\frac{3}{16}\frac{(m-1)^{2}}{2^{2j}}\xi
\le \tau-\frac14\xi^{3}\le3\frac{(m+1)^{2}}{2^{2j}}\xi
&&\text{if}\ m=2,3.\\
\end{aligned}
\right.\label{cc11}
\end{eqnarray}
Hence for any $j\in{{\mathbb Z}}$ and $m=-3,-2,2,3$, we have
\begin{eqnarray*}
\sum_{k\ge\max\{0,-m\}}{\bf 1}_{\widetilde{A}_{j,k,k+m}}(\tau,\xi)
\le{\bf 1}_{C_{j,m}}(\tau,\xi) \qquad a.e.,
\end{eqnarray*}
where 
$C_{j,m}=\left\{(\tau,\xi)\left|\xi\ge0,
\right.\tau\ \text{satisfies}\ (\ref{cc11})\right\}$.
Therefore
\begin{eqnarray}
\sum_{j\in{{\mathbb Z}}}\sum_{k\ge\max\{0,-m\}}{\bf 1}_{\widetilde{A}_{j,k,k+m}} (\tau,\xi)\le3\qquad a.e.
\label{tz1}
\end{eqnarray}
for each $m=-3,-2,2,3$. Hence we obtain (\ref{orthg}) for $X=A$. 

By (\ref{dd2}), we find
\begin{eqnarray*}
\widetilde{B}_{j,k,k+m}
\subset\left\{(\tau,\xi)\ \left|\ \frac{-m-1}{2^{j}}\le\xi\le
\frac{-m+1}{2^{j}},
\right.\tau\ \text{satisfies}\ (\ref{cc21})\right\},
\end{eqnarray*}
where 
\begin{eqnarray}
\left\{
\begin{aligned}
&\frac{3}{16}\frac{(2k+m)^{2}}{2^{2j}}\xi
\le \tau-\frac14\xi^{3}\le3\frac{(2(k+1)+m)^{2}}{2^{2j}}\xi
\\
&\qquad\qquad\qquad\qquad\qquad\qquad\qquad\qquad\text{if}\ m=-3,-2,\\
&\frac{3}{16}\frac{(2(k+1)+m)^{2}}{2^{2j}}\xi
\le \tau-\frac14\xi^{3}\le3\frac{(2k+m)^{2}}{2^{2j}}\xi
\\
&\qquad\qquad\qquad\qquad\qquad\qquad\qquad\qquad\text{if}\ m=2,3.
\end{aligned}
\right.\label{cc21}
\end{eqnarray}
Therefore, for any $j\in{{\mathbb Z}}$ and $m=-3,-2,2,3$ we have
\begin{eqnarray*}
\sum_{k\ge\max\{0,m\}}{\bf 1}_{\widetilde{B}_{j,k,k+m}} (\tau,\xi) \le 3 {\bf 1}_{D_{j,m}} (\tau,\xi)
\qquad a.e.,
\end{eqnarray*}
where $D_{j,m}$ are given by 
\begin{eqnarray*}
D_{j,m}
&=&\left\{(\tau,\xi)\in\rre^{2}\ \left|\ \frac{-m-1}{2^{j}}\le\xi\le
\frac{-m+1}{2^{j}}\right.\right\}.
\nonumber
\end{eqnarray*}
This implies
\begin{eqnarray}
\sum_{j\in{{\mathbb Z}}}\sum_{k\ge\max\{0,-m\}}{\bf 1}_{\widetilde{B}_{j,k,k+m}} (\tau,\xi) \le3
\qquad a.e.
\label{tz2}
\end{eqnarray}
for any $m=-3,-2,2,3$. Then we obtain (\ref{orthg}) for $X=B$. 
This completes the proof of Proposition \ref{orthgnal}.
\end{proof}

\begin{remark}\label{rem:si}
For closed domain $R\subset\rre^{2}$ and $\lambda>0$, we define
\begin{eqnarray*}
R'_{+\lambda}=\{(\tau+\tau',\xi+\xi')|(\tau,\xi)\in R,\ -\lambda\le\tau'
\le\lambda,-\lambda\le\xi'\le\lambda\},
\end{eqnarray*}
which is an enlargement both in $\tau$- and $\xi$-directions.
Further, we define $\tilde{A'}_{j,k,\ell}$  $\tilde{B'}_{j,k,\ell}$ by 
\begin{equation*}
	\widetilde{A'}_{j,k,\ell} = (A_{j,k,\ell})'_{+\frac{k}{100\times2^{3j}}},\quad
	\widetilde{B'}_{j,k,\ell} = (B_{j,k,\ell})'_{+\frac{k}{100\times2^{3j}}}.
\end{equation*}
If we are able to show the almost orthogonality properties of the two families 
$\{\tilde{A'}_{j,k,\ell}\}$ and 
$\{\tilde{B'}_{j,k,\ell}\}$, then we will be able to 
obtain Theorems \ref{thm:S} and \ref{thm:T} with $\sigma=0$
by means of \cite[Lemma 6.1]{TVV}, in essentially the same spirit as in \cite{BV}.
\end{remark}

Next we show the bounds for the Fourier multipliers $P_{\tilde{A}_{j,k,\ell}}$ 
and $P_{\tilde{B}_{j,k,\ell}}$ defined by  
\begin{eqnarray*}
P_{\tilde{X}_{j,k,\ell}}:=P_{X_{j,k,\ell},\frac{k}{100\times2^{3j}}}\qquad\text{for}\ X=A,B, 
\end{eqnarray*}
where $P_{X_{j,k,\ell},\frac{k}{100\times2^{3j}}}$ is given by (\ref{def:PR}). 
Since $P$ is a frequency cutoff which is smooth only in $\tau$-direction, we are not free from
a small \emph{loss} in the exponent of $x$.

\begin{proposition}[Bounds for multiplier]\label{boundness} 
Let $X=A$ or $B$. 

\vskip1mm
\noindent
(i) Let $\si>0$, $1/(1-\si)\le p\le\infty$ and  
$1\le q\le\infty$. 
Let $p_{\si}$ be given by  
\begin{eqnarray}
\frac1{p_{\si}} = \frac1p+\si.\label{ppt}
\end{eqnarray}
Then, there exists a positive constant $C$ depending only on $p,q$ 
such that for any $(j,k,\ell)\in\Lambda$, the inequality 
\begin{eqnarray}
	\|P_{\tilde{X}_{j,k,\ell}}F\|_{L_{x}^{p}L_{t}^{q}}
	\le C2^{-j\si}\|F\|_{L_{x}^{p_{\si}}L_{t}^{q}}
	\label{bound}
\end{eqnarray}
holds for any $F\in L_{x}^{p_{\si}}L_{t}^{q}$.

\vskip1mm
\noindent
(ii) Let $\si>0$, $1\le p\le\infty$ and  
$1/(1-\si)\le q\le\infty$. 
Let $q_{\si}$ be given by  
\begin{eqnarray*}
\frac1{q_{\si}} = \frac1q+\si.
\end{eqnarray*}
Then, there exists a positive constant $C$ depending only on $p,q$ 
such that for any $(j,k,\ell)\in\Lambda$, the inequality 
\begin{eqnarray}
	\|P_{\tilde{X}_{j,k,\ell}}F\|_{L_{t}^{p}L_{x}^{q}}
	\le C2^{-j\si}\|F\|_{L_{t}^{p}L_{x}^{q_{\si}}}
	\label{bound3}
\end{eqnarray}
holds for any $F\in L_{t}^{p}L_{x}^{q_{\si}}$.
\end{proposition}

\begin{proof}[Proof of Proposition \ref{boundness}] 
We show the inequality (\ref{bound}) only since the proof of 
(\ref{bound3}) is similar. 
Consider a set of the form
\begin{equation*}
R:=\left\{(\tau,\xi)\left|a\le\xi\le b,c\xi\le\tau-\frac14\xi^{3}\le 
d\xi\right.\right\}
\end{equation*}
with constants $a,b,c,d\in\rre$.
To prove (\ref{bound}), it suffices to show the following assertion: 
For $a,b,c,d\in\rre$ satisfying the relations 
\begin{equation*}
(b+a)(d-c)\le200\frac{k}{2^{3j}},\qquad b-a\le\frac{2}{2^{j}}
\end{equation*}
and for $\lambda=k/2^{3j}$, it holds that
\begin{eqnarray}
	\|P_{R,\lambda}F\|_{L_{x}^{p}L_{t}^{q}}
	\le C2^{-j\si}\|F\|_{L_{x}^{p_{\si}}L_{t}^{q}}.
	\label{bound2}
\end{eqnarray}
To prove (\ref{bound2}), 
we first evaluate the inverse Fourier transform of the characteristic 
function ${\bf 1}_{{R}_{+\frac{\lambda}{2}}}$. A direct calculation shows 
\begin{eqnarray*}
{{\mathcal F}}^{-1}[{\bf 1}_{{R}_{+\frac{\lambda}{2}}}](t,x)
&=&
\int_{a}^{b}\left(\int_{\frac14\xi^{3}+c\xi-\frac{\lambda}{2}
}^{\frac14\xi^{3}+d\xi+\frac{\lambda}{2}}e^{ix\xi+it\tau}d\tau\right)d\xi\\
&=&
\int_{a}^{b}e^{i(x+ft)\xi+\frac14it\xi^{3}}\left(\int_{-\frac{d-c}{2}\xi-\frac{\lambda}{2}
}^{\frac{d-c}{2}\xi+\frac{\lambda}{2}}e^{it\tau}d\tau\right)d\xi,
\end{eqnarray*}
where $f=(d+c)/2$. We easily see 
\begin{eqnarray}
\lefteqn{|{{\mathcal F}}^{-1}[{\bf 1}_{{R}_{+\frac{\lambda}{2}}}](t,x)|
\le \int_{a}^{b}\left(\int_{-\frac{d-c}{2}\xi-\frac{\lambda}{2}
}^{\frac{d-c}{2}\xi+\frac{\lambda}{2}}d\tau\right)d\xi}\label{q1}\\
&\le&\frac12(b^{2}-a^{2})(d-c)+\lambda(b-a)
\le C\frac{k}{2^{4j}}.\nonumber
\end{eqnarray}
On the other hand, we evaluate ${{\mathcal F}}^{-1}[
{\bf 1}_{{R}_{+\frac{\lambda}{2}}}]$ 
by using the method of stationary phase. We 
rewrite ${{\mathcal F}}^{-1}[{\bf 1}_{{R}_{+\frac{\lambda}{2}}}]$ 
as 
\begin{eqnarray*}
\lefteqn{{{\mathcal F}}^{-1}[{\bf 1}_{{R}_{+\frac{\lambda}{2}}}](t,x)}\\
&=&
\int_{a}^{b}e^{i\{x+(\frac34e^{2}+f)t\}\xi}
e^{\frac14it(\xi^{3}-3e^{2}\xi)}\left(\int_{-\frac{d-c}{2}\xi-\frac{\lambda}{2}
}^{\frac{d-c}{2}\xi+\frac{\lambda}{2}}e^{it\tau}d\tau\right)d\xi,
\end{eqnarray*}
where $e=(a+b)/2$. Using the identity 
\begin{eqnarray*}
e^{i\{x+(\frac34e^{2}+f)t\}\xi}=\frac{\pt_{\xi}e^{i\{x+(\frac34e^{2}+f)t\}\xi}
}{i\{x+(\frac34e^{2}+f)t\}}
\end{eqnarray*}
and integrating by parts, we obtain
\begin{eqnarray*}
\lefteqn{{{\mathcal F}}^{-1}[{\bf 1}_{{R}_{+\frac{\lambda}{2}}}](t,x)}\\
&=&
\frac{1}{i\{x+(\frac34e^{2}+f)t\}}
\left[e^{ix\xi+\frac14it\xi^{3}}\left(\int_{-\frac{d-c}{2}\xi-\frac{\lambda}{2}
}^{\frac{d-c}{2}\xi+\frac{\lambda}{2}}e^{it\tau}d\tau\right)
\right]_{a}^{b}\\
& &-
\frac{3t}{4\{x+(\frac34e^{2}+f)t\}}
\int_{a}^{b}(\xi^{2}-e^{2})e^{ix\xi+\frac14it\xi^{3}}\left(\int_{-\frac{d-c}{2}\xi-\frac{\lambda}{2}
}^{\frac{d-c}{2}\xi+\frac{\lambda}{2}}e^{it\tau}d\tau\right)d\xi\\
& &-
\frac{d-c}{2i\{x+(\frac34e^{2}+f)t\}}
\int_{a}^{b}e^{ix\xi+\frac14it\xi^{3}}
(e^{it(\frac{d-c}{2}\xi+\frac{\lambda}{2})}+e^{it(-\frac{d-c}{2}\xi-\frac{\lambda}{2})})d\xi.
\end{eqnarray*}
Hence
\begin{eqnarray}
|{{\mathcal F}}^{-1}[{\bf 1}_{{R}_{+\frac{\lambda}{2}}}](t,x)|
&\le&C\frac{(b+a)(d-c)+(b-a)^{3}+\lambda}{|x+(\frac34e^{2}+f)t|}
\label{q2}\\
&\le&\frac{k}{2^{3j}}\frac{C}{|x+(\frac34e^{2}+f)t|}.
\nonumber
\end{eqnarray}
By (\ref{q1}) and (\ref{q2}), we have
\begin{eqnarray*}
|{{\mathcal F}}^{-1}[{\bf 1}_{{R}_{+\frac{\lambda}{2}}}](t,x)|
\le C\frac{k}{2^{3j}}\times
\left\{
\begin{aligned}
&\frac{1}{2^{j}}
&&\text{if}\ |x+(3e^{2}/4+f)t|\le 2^{j},\\
&\frac{1}{|x+(\frac34e^{2}+f)t|}
&&\text{if}\ |x+(3e^{2}/4+f)t|\ge 2^{j}.
\end{aligned}
\right.
\end{eqnarray*}
Therefore we find 
\begin{eqnarray*}
\|{{\mathcal F}}^{-1}[{\bf 1}_{{R}_{+\frac{\lambda}{2}}}]
\|_{L_{x}^{r_{\si}}}
\le C\frac{k}{2^{3j}}2^{-j\si},
\end{eqnarray*}
where $r_{\si}$ satisfies
$-\si=1/r_{\si}-1.$
Combining the above inequality with the Young and Minkowski inequalities, 
we obtain
\begin{eqnarray*}
\lefteqn{\|P_{R,\lambda}F\|_{L_{x}^{p}L_{t}^{q}}}\\
&\le&
C\|\F^{-1}[\varphi]({\lambda}t/{2})\F^{-1}[
{\bf 1}_{{R}_{+\frac{\lambda}{2}}}]
\|_{L_{x}^{r_{\si}}L_{t}^{1}}\|F\|_{L_{x}^{p_{\si}}L_{t}^{q}}\\
&\le&
C\|\F^{-1}[\varphi](\lambda t/2)
\F^{-1}[{\bf 1}_{{R}_{+\frac{\lambda}{2}}}]
\|_{L_{t}^{1}L_{x}^{r_{\si}}}\|F\|_{L_{x}^{p_{\si}}L_{t}^{q}}
\\
&\le&C2^{-j\si}\|F\|_{L_{x}^{p_{\si}}L_{t}^{q}}.
\end{eqnarray*}
This proves (\ref{bound2}) and completes the proof.	
\end{proof}

\subsection{Proof of main results}

We now prove Theorem \ref{thm:S} (\ref{qwe}). 
We may suppose that $p\neq q$ because the case $p=q$ is already proved in 
\cite[Theorem B.1]{MS2}.
We evaluate $I_2$ in (\ref{mob}) only since the estimate for $I_1$ in (\ref{mob}) 
is similar.

To evaluate $I_{2}$, we first show that
\begin{eqnarray}
\lefteqn{\qquad\|\sum_{\tau_{k}^{j}\in{{\mathcal D}}}
\sum_{\tau_{\ell}^{j}:\tau_{\ell}^{j}\sim \tau_{k}^{j}}
|\pt_{x}|^{s}e^{-t\pt_{x}^{3}}f_{\tau_{k}^{j}}
\overline{|\pt_{x}|^{s}e^{-t\pt_{x}^{3}}f_{\tau_{\ell}^{j}}}
\|_{L_{x}^{\frac{p}{2}}L_{t}^{\frac{q}{2}}}}
\label{s1}\\
&\le&C\biggl(\sum_{\tau_{k}^{j}\in{{\mathcal D}}}
\sum_{\tau_{\ell}^{j}:\tau_{\ell}^{j}\sim \tau_{k}^{j}}
|\tau^{j}_{k}|^{\delta\si}
\||\pt_{x}|^{s}e^{-t\pt_{x}^{3}}f_{\tau_{k}^{j}}
\overline{|\pt_{x}|^{s}e^{-t\pt_{x}^{3}}f_{\tau_{\ell}^{j}}}
\|_{L_{x}^{\frac{p_{\si}}{2}}L_{t}^{\frac{q}{2}}}^{\frac{\delta}{2}}
\biggl)^{\frac{2}{\delta}},
\nonumber
\end{eqnarray}
where the exponent $p_{\si}$ is given by (\ref{ppt}).
The inequality (\ref{s1}) follows from
\begin{eqnarray}
\| T F \|_{
L_{x}^{\frac{p}{2}}L_{t}^{\frac{q}{2}}}
\le C\|2^{-2j\si}
F\|_{\dot{\ell}_{j}^{\frac{\delta}{2}}{\ell}_{k}^{\frac{\delta}{2}}
L_{x}^{\frac{p_{\si}}{2}}L_{t}^{\frac{q}{2}}}
\label{s2}
\end{eqnarray}
for function $F=F(j,k,t,x)$,
where $T$ is defined by
\begin{eqnarray*}
	(T F)(t,x) &=& \sum_{(j,k,\ell)\in\Lambda}P_{\tilde{A}_{j,k,\ell}} 
	F(j,k, t,x)\\
	&=&
	\sum_{m=-3,-2,2,3}\sum_{j\ge0}\sum_{k\ge\max\{0,-m\}}P_{\tilde{A}_{j,k,k+m}} 
	F(j,k, t,x),\\
\end{eqnarray*}
where $\tilde{A}_{j,k,\ell}$ is given by (\ref{tAB}),  
$\|a_{j}\|_{\dot{\ell}_{j}^{\delta}}
=(\sum_{j\in{{\mathbb Z}}}|a_{j}|^{\delta})^{1/\delta}$, and 
$\|a_{k}\|_{\ell_{k}^{\delta}}=(\sum_{k\in{{\mathbb Z}}_{+}}
|a_{k}|^{\delta})^{1/\delta}$. 
Indeed, taking
$$F(j,k,t,x)=\sum_{\tau_{\ell}^{j}:\tau_{\ell}^{j}\sim \tau_{k}^{j}}
|\pt_{x}|^{s}e^{-t\pt_{x}^{3}}f_{\tau_{k}^{j}}
\overline{|\pt_{x}|^{s}e^{-t\pt_{x}^{3}}f_{\tau_{\ell}^{j}}}$$
in \eqref{s2} and using triangle inequality, we have (\ref{s1}). 

Let us prove \eqref{s2}. 
The Plancherel identity and the almost orthogonality 
(Proposition \ref{orthgnal} (\ref{orthg})) imply 
\begin{eqnarray}
\|T F \|_{L_{x}^{2}L_{t}^{2}}
\le C\| F \|_{\dot{\ell}_{j}^{2}{\ell}_{k}^{2}
L_{x}^{2}L_{t}^{2}}. \label{ss2}
\end{eqnarray}
On the other hand, the triangle inequality and Proposition \ref{boundness} 
yield 
\begin{eqnarray}
\| TF \|_{L_{x}^{P}L_{t}^{Q}}
\le C\| 2^{-2j\theta\si}F \|_{\dot{\ell}_{j}^{1}{\ell}_{k}^{1}
L_{x}^{P_{\si}}L_{t}^{Q}},\label{ss3}
\end{eqnarray}
where
\begin{equation*}
\(\theta,\frac1P,\frac1Q,
\frac{1}{P_{\si}}\)
=\left\{
\begin{aligned}
&\left(\frac{p}{p-4},0,\frac{2(p-q)}{q(p-4)},\frac{2p}{p-4}\si
\right),
&& \text{if}\ p>q,\\
&\left(\frac{q}{q-4},\frac{2(q-p)}{p(q-4)},0,\frac{2(q-p)}{p(q-4)}+
\frac{2q}{q-4}\si
\right),
&& \text{if}\ p<q.
\end{aligned}
\right.
\end{equation*}
Interpolating (\ref{ss2}) and (\ref{ss3}), we obtain \eqref{s2}
with
\[
	\frac1\delta = \frac12- \min\( \frac{1}{p}, \frac1q\)
	=\frac12- \frac1{\max(p,q)}.
\]

To show (\ref{qwe}), we consider the two cases: 
$p<q$ and $p>q$.

\vskip1mm
\noindent
{\bf Case: $p<q$}. 
Since $\tau_k^j\sim\tau_\ell^j$ implies $k\neq0$ or $\ell\neq0$
\footnote{Notice that the case $k=0$ and $\ell=0$ never happens.}, 
we may assume $\ell\neq0$ in which case $\dist(0,\tau_{\ell}^j) \neq0$.

By an argument used in \cite[Proposition B.1]{MS2}, 
we find
\begin{eqnarray}
\lefteqn{\qquad\||\pt_{x}|^{\frac{1}{p_{\sigma}}+\si}e^{-t\pt_{x}^{3}}
f_{\tau_{k}^{j}}
\overline{|\pt_{x}|^{s}e^{-t\pt_{x}^{3}}
f_{\tau_{\ell}^{j}}}
\|_{L_{t,x}^{\frac{p_{\si}}{2}}}}\label{e12}\\
& &\qquad\le C\dist(0,\tau_{\ell}^j)^{s-\frac{1}{p_{\si}}-\si}
|\tau_{k}^{j}|^{-\frac{2}{p_{\si}}}
\||\xi|^{\si}\hat{f}_{\tau_{k}^{j}}\|_{L_{\xi}^{(\frac{p_{\si}}{2})'}}
\||\xi|^{\si}\hat{f}_{\tau_{\ell}^{j}}\|_{L_{\xi}^{(\frac{p_{\si}}{2})'}},
\nonumber
\end{eqnarray}
where we used the inequality $p_{\si}/2\ge2$. 
Further, \cite[Proposition 2.1]{MS1} yields 
\begin{eqnarray}
\lefteqn{\|
|\pt_{x}|^{-\frac{1}{p_{\sigma}}+\si}e^{-t\pt_{x}^{3}}f_{\tau_{k}^{j}}
\overline{|\pt_{x}|^{s}e^{-t\pt_{x}^{3}}f_{\tau_{\ell}^{j}}}
\|_{L_{x}^{\frac{p_{\si}}2}L_{t}^{\I}}}
\label{e22}\\
&\le&
\|
|\pt_{x}|^{-\frac{1}{p_{\sigma}}+\si}e^{-t\pt_{x}^{3}}
f_{\tau_{k}^{j}}
\|_{L_{x}^{p_{\si}}L_{t}^{\I}}
\|
|\pt_{x}|^{s}
e^{-t\pt_{x}^{3}}f_{\tau_{\ell}^{j}}
\|_{L_{x}^{p_{\si}}L_{t}^{\I}}
\nonumber\\
&\le&C \dist(0,\tau_{\ell}^j)^{s+\frac{1}{p_{\si}}-\si} 
\||\xi|^{\si}\hat{f}_{\tau_{k}^{j}}\|_{L_{\xi}^{(\frac{p_{\si}}{2})'}}
\||\xi|^{\si}\hat{f}_{\tau_{\ell}^{j}}\|_{L_{\xi}^{(\frac{p_{\si}}{2})'}},
\nonumber
\end{eqnarray}
where we also used the inequality $p_{\si}\ge4$. 
Hence, it holds from the Stein interpolation for mixed norm 
(see \cite[Section 7, Theorem 1]{BP}), (\ref{e12}) and (\ref{e22}) 
that
\begin{eqnarray}
\lefteqn{\|
|\pt_{x}|^{s}e^{-t\pt_{x}^{3}}f_{\tau_{k}^{j}}
\overline{|\pt_{x}|^{s}e^{-t\pt_{x}^{3}}f_{\tau_{\ell}^{j}}}
\|_{L_{x}^{\frac{p_{\si}}{2}}L_{t}^{\frac{q}{2}}}}
\label{sss4}\\
& &\le C|\tau_{k}^{j}|^{-\frac2q}
\||\xi|^{\si}\hat{f}_{\tau_{k}^{j}}\|_{L_{\xi}^{(\frac{p_{\si}}2)'}}
\||\xi|^{\si}\hat{f}_{\tau_{\ell}^{j}}\|_{L_{\xi}^{(\frac{p_{\si}}2)'}}.
\nonumber
\end{eqnarray}
Collecting (\ref{s1}) and (\ref{sss4}), we obtain 
\begin{eqnarray}
\|I_{2}\|_{L_{x}^{\frac{p}{2}}L_{t}^{\frac{q}{2}}}\le C\|
|\pt_{x}|^{\si}f\|_{\hat{M}^{\beta}_{\gamma,\delta}}^{2}.
\label{s3}
\end{eqnarray}

\vskip1mm
\noindent
{\bf Case: $p>q$}. As in the previous case, we may assume $\ell\neq0$. 
By an argument used in \cite[Proposition B.1]{MS2}, 
we find
\begin{eqnarray}
\lefteqn{\qquad\||\pt_{x}|^{\frac{p_{\si}+q
}{2p_{\si}q}+\si}e^{-t\pt_{x}^{3}}
f_{\tau_{k}^{j}}
\overline{|\pt_{x}|^{s}e^{-t\pt_{x}^{3}}
f_{\tau_{\ell}^{j}}}
\|_{L_{t,x}^{\frac{p_{\si}q}{p_{\si}+q}}}}\label{e1}\\
& &\quad\le C\dist(0,\tau_{\ell}^j)^{s-\frac{p_{\si}+q
}{2p_{\si}q}-\si}|\tau_{k}^{j}|^{-\frac{p_{\si}+q
}{p_{\si}q}}
\||\xi|^{\si}\hat{f}_{\tau_{k}^{j}}\|_{L_{\xi}^{(\frac{p_{\si}q}{p_{\si}+q})'}}
\||\xi|^{\si}\hat{f}_{\tau_{\ell}^{j}}\|_{L_{\xi}^{(\frac{p_{\si}q}{p_{\si}+q})'}}, 
\nonumber
\end{eqnarray}
where we used the inequality $p_{\si}q/(p_{\si}+q)\ge2$. 
On the other hand, \cite[Proposition 2.1]{MS1} yields 
\begin{eqnarray}
\lefteqn{\|
|\pt_{x}|^{\frac{2(p_{\si}+q)}{
p_{\si}q}+\si}e^{-t\pt_{x}^{3}}f_{\tau_{k}^{j}}
\overline{|\pt_{x}|^{s}e^{-t\pt_{x}^{3}}f_{\tau_{\ell}^{j}}}
\|_{L_{x}^{\I}L_{t}^{\frac{p_{\si}q}{2(p_{\si}+q)}}}}
\label{e2}\\
&\le&
\|
|\pt_{x}|^{\frac{2(p_{\si}+q)}{
p_{\si}q}+\si}e^{-t\pt_{x}^{3}}
f_{\tau_{k}^{j}}
\|_{L_{x}^{\I}L_{t}^{\frac{p_{\si}q}{p_{\si}+q}}}
\|
|\pt_{x}|^{s}
e^{-t\pt_{x}^{3}}f_{\tau_{\ell}^{j}}
\|_{L_{x}^{\I}L_{t}^{\frac{p_{\si}q}{p_{\si}+q}}}
\nonumber\\
&\le&C \dist(0,\tau_{\ell}^j)^{s-\frac{2(p_{\si}+q)}{
p_{\si}q}-\si} 
\||\xi|^{\si}\hat{f}_{\tau_{k}^{j}}
\|_{L_{\xi}^{(\frac{p_{\si}q}{p_{\si}+q})'}}
\||\xi|^{\si}\hat{f}_{\tau_{\ell}^{j}}
\|_{L_{\xi}^{(\frac{p_{\si}q}{p_{\si}+q})'}}.
\nonumber
\end{eqnarray}
Combining Stein interpolation for mixed norm with 
(\ref{e1}) and (\ref{e2}), we obtain
\begin{eqnarray}
\lefteqn{\|
|\pt_{x}|^{s}e^{-t\pt_{x}^{3}}f_{\tau_{k}^{j}}
\overline{|\pt_{x}|^{s}e^{-t\pt_{x}^{3}}f_{\tau_{\ell}^{j}}}
\|_{L_{x}^{\frac{p\si}{2}}L_{t}^{\frac{q}{2}}}}
\label{ss4}\\
& &\qquad\le C|\tau_{k}^{j}|^{-\frac2{p_{\si}}}
\||\xi|^{\si}\hat{f}_{\tau_{k}^{j}}\|_{L_{\xi}^{(\frac{p_{\si}q}{p_{\si}+q})'}}
\||\xi|^{\si}\hat{f}_{\tau_{\ell}^{j}}\|_{L_{\xi}^{(\frac{p_{\si}q}{p_{\si}+q})'}}.
\nonumber
\end{eqnarray}
Collecting (\ref{s1}) and (\ref{ss4}), we obtain (\ref{s3}). 
For the sub-case $p_{\si}<q$, the similar argument as that in the case $p<q$ 
yields (\ref{s3}). 

In a similar way, we obtain
\begin{eqnarray}
\|I_{1}\|_{L_{x}^{\frac{p}{2}}L_{t}^{\frac{q}{2}}}\le 
C\||\pt_{x}|^{\si}f\|_{\hat{M}^{\alpha}_{\beta,\gamma}}^{2}.
\label{s35}
\end{eqnarray}
Combining (\ref{mob}) with (\ref{s3}) and (\ref{s35}), we obtain (\ref{qwe}). 
This completes the proof of Theorem \ref{thm:S}.

%
%

\section{Application to well-posedness}

In this section we prove local and global well-posedness for (\ref{gKdV}) (Theorems \ref{thm:lwp} 
and \ref{thm:SDS}).
To this end, we consider integral form of \eqref{gKdV}:
\begin{eqnarray}
u(t) = e^{-(t-t_0)\pt_x^3} u_0
+ \mu \int_{t_0}^t e^{-(t-s)\pt_x^3}\pt_x (|u|^{2\alpha}u)(s) ds.
\label{gKdV2}
\end{eqnarray}
Let $\sigma>0$ and define $\beta$ by $1/\beta=1/\alpha + \sigma$ 
as in Theorem \ref{thm:lwp}. 
For an interval $I\subset \R$, we introduce function spaces $L(I)$, $M(I)$, $S(I)$, and $D_\si(I)$ 
as follows:
\begin{eqnarray*}
	L(I)&:=& \left\{ u\in \Sch'(I \times \R) \left|
	\norm{u}_{L(I)}:= \norm{|\pt_x|^{\frac1{\alpha}}u}_{L^{5\alpha}_{x}(\R ; L^{\frac{5\alpha}{3}}_t(I))} <\I
	\right.\right\},\\
	M(I)&:=& \left\{ u\in \Sch'(I \times \R) \left|
	\norm{u}_{M(I)}:= \norm{|\pt_x|^{\frac1{2\alpha}}u}_{L^{\frac{10\alpha}{3}}_{x}(\R ; L^{\frac{5\alpha}{2}}_t(I))} <\I
	\right.\right\},\\
	S(I)&:=& \left\{ u\in \Sch'(I \times \R) \left|
	\norm{u}_{S(I)}:= \norm{u}_{L^{\frac{5\alpha}2}_{x}(\R ; L^{5\alpha}_{t}(I)) } <\I
	\right.\right\},\\
	D_\sigma(I)&:=& \left\{ u\in \Sch'(I \times \R) \left|
	\norm{u}_{D_\sigma(I)}:= \norm{|\pt_x|^{\sigma+\frac1{3\beta}}u}_{L^{3\beta}_{t,x}(I\times \R)} <\I
	\right.\right\}
.
\end{eqnarray*}
For an interval $I\subset \R$,
we say a function $u\in M(I) \cap S(I)$ is a solution to \eqref{gKdV} on $I$
if $u$ satisfies \eqref{gKdV2} in the $M(I) \cap S(I)$ sense. 
We modify a well-posedness result in \cite{MS1}.

\begin{lemma}\label{lem:lwp_pre}
Let $5/3<\alpha\le 20/9$. 
Denote by $Z(I)$ either $L(I)$ or $M(I)$.
Let $t_0 \in \R$ and $I$ be an interval with $t_{0}\in I$.
Then, there exists a universal constant $\delta>0$ such that 
if  a tempered distribution $u_0$ and an interval $I\ni t_0$ satisfy
\[
	\eta_0=\eta_0(I;u_0,t_0):=\norm{e^{-(t-t_0)\d_x^{3}} u_0}_{S(I)}
	+ \norm{e^{-(t-t_0)\d_x^{3}}u_0}_{Z(I)}
	\le \delta,
\]
then there exists a unique solution $u(t)$ on $I$ to (\ref{gKdV}) satisfying
\[
	\norm{u}_{S(I)}+ \norm{u}_{Z(I)} \le 2\eta.
\]
Furthermore, the solution satisfies $u(t)- e^{-(t-t_0)\d_x^{3}} u_0 \in C(I;\hat{L}^\alpha)$.
\end{lemma}
We omit the proof.
Instead, we remark that
$L(I):=X(I;1/\alpha,\alpha)$ and $M(I):=X(I;1/(2\alpha),\alpha)$ 
in the notation of \cite[Definition 1.1]{MS1},
and that the pairs $(s,r)=(1/(2\alpha),\alpha), (1/\alpha,\alpha)$ are acceptable and conjugate-acceptable
in the sense of \cite[Definitions 1.1 and 3.1]{MS1} as long as 
$10/7 < \alpha < 8/3$ if $Z=L$ and
$3/2 \le \alpha < 7/3$ if $Z=M$, which
 are weaker than our assumption.

As a corollary of 
Theorem \ref{thm:S} and Lemma \ref{lem:lwp_pre}, we obtain an
existence result.
\begin{corollary}\label{cor:existence}
Let $\alpha$, $\si$, $\beta$, $\gamma$ and $\delta$ 
satisfy the assumption of Theorem \ref{thm:lwp}.
Then, for any $u_0 \in |\d_x|^{-\sigma} \hat{M}^{\beta}_{\delta,\gamma}$ and $t_0 \in \R$
there exists an interval $I \subset \R$, $I\ni t_0$ such that there exists a unique solution $u(t)$ on $I$ to (\ref{gKdV}).
The solution belongs to $C(I;|\d_x|^{-\sigma} \hat{M}^{\beta}_{\gamma,\delta}+\hat{L}^{\alpha})$.
\end{corollary}

\begin{proof}[Proof of Corollary \ref{cor:existence}] 
One sees from Theorem \ref{thm:S} that if $\alpha > 8/5$ and $0<\sigma \le 1/4-2/(5\alpha)$ then
\begin{equation}\label{eq:linear1}
	\norm{e^{-(t-t_0)\d_x^3}u_0}_{L(\R)} + \norm{e^{-(t-t_0)\d_x^3}u_0}_{S(\R)}
	\le C \norm{ |\d_x|^\sigma  u_0}_{\hat{M}^\beta_{\gamma,\delta}} <\I.
\end{equation}
Hence, there exists an open neighborhood $I \subset \R $ of $t_0$ such that
$\eta_0(I;u_0,t_0)\le \delta$, where $\delta$ and $\eta_0$ are defined in Lemma \ref{lem:lwp_pre}.
Since $u(t)- e^{-(t-t_0)\d_x^{3}} u_0 \in C(I;\hat{L}^\alpha)$ and $|\d_x|^\sigma e^{-(t-t_0)\d_x^{3}} u_0 \in C(I;\hat{M}^\beta_{\gamma,\delta})$, we obtain the result.
\end{proof}
\begin{proof}[Proof of Theorem \ref{thm:lwp}]
To prove the theorem, it suffices to show that 
$u(t)- e^{-(t-t_0)\d_x^{3}} u_0 \in C(I, |\d_x|^{-\sigma} \hat{M}^{\beta}_{\gamma,\delta})$. 
We mimic the argument in \cite{MS1,MS2}.

We infer from the diagonal refined estimate and
the inhomogeneous Strichartz' estimate 
\cite[Proposition 2.5]{MS1} that
\begin{eqnarray}
\|u\|_{D_\sigma(I)}
\le C \norm{|\d_x|^\sigma u_0}_{\hat{M}^{\beta}_{\frac{3\beta}2,\frac{6\beta}{3\beta-2}}} + C
\| |u|^{2\alpha}u\|_{N_\sigma(I)},\label{k1}
\end{eqnarray}
where
\begin{eqnarray*}
\|f\|_{N_\sigma(I)}:=\||\pt_{x}|^{\frac{1}{3\beta}+\si}f\|_{L_{x}^{p(N_\sigma)}L_{t}^{q(N_\sigma)}},
\end{eqnarray*}
and $(p(N_\sigma),q(N_\sigma))$ is given by
\begin{eqnarray*}
\frac{2}{p(N_\sigma)}+\frac{1}{q(N_\sigma)}=\frac{1}{\beta}+2,\qquad
-\frac{1}{p(N_\sigma)}+\frac{2}{q(N_\sigma)}=\frac{1}{3\beta}.
\end{eqnarray*}
Note that the pair $(s,r)=(1/(3\beta),\beta)$ is acceptable and conjugate-acceptable
in the sense of \cite[Definitions 1.1 and 3.1]{MS1} if $5/3\le\beta<20/9$. 
To choose such $\beta$, we need the restrictions $5/3<\alpha\le20/9$ and $0<\sigma \le3/5-1/\alpha$.
We then apply the Leibniz rule for the fractional order derivatives 
\cite[Lemma 3.4]{MS1} to obtain 
\begin{eqnarray}
\||u|^{2\alpha} u\|_{N_\sigma(I)}\le C\|u\|_{S(I)}^{2\alpha}\|u\|_{D_\sigma(I)}.
\label{k2}
\end{eqnarray}
We divide $I$ into subintervals $\{I_j\}_{j=1}^J$ so that $\norm{u}_{S(I_j)}$ is small.
Then, it follows from \eqref{k1} and \eqref{k2} that $\|u\|_{D_\sigma (I_j)} <\I$ for each subinterval, showing $u \in D_\sigma (I)$.
Then, we conclude from the inhomogeneous Strichartz' estimate \cite[Proposition 2.5]{MS1} 
that 
\[
	|\pt_{x}|^{\sigma}\int_{t_0}^t e^{-(t-s)\pt_x^3}\pt_x (|u|^{2\alpha}u) ds.
	\in C(I;\hat{L}^{\beta}) 
\]
This completes Theorem \ref{thm:lwp} since $\hat{L}^{\beta} \hookrightarrow \hat{M}^{\beta}_{\gamma,\delta}$. 
\end{proof}

\begin{remark}
Let us summarize our assumption on the local well-posedenss result.
For the estimate \eqref{k1}, we need $5/3<\alpha\le20/9$ and $0<\sigma \le 3/5-1/\alpha$.
Further, the restriction $\sigma \le 1/4-2/(5\alpha)$ comes from \eqref{eq:linear1}.
The possible ranges of $\gamma$, $\delta$ are also ruled by \eqref{eq:linear1}.
\end{remark}

\subsection{Criterion for blowup and scattering}
We next show standard criterion for finite-time blowup and scattering.
These are essentially the same as \cite[Theorems 1.8 and 1.9]{MS1}.
Let $I_{\max} = (T_{\min},T_{\max})$ be the maximal interval.
\begin{theorem}[Criterion of finite-time blowup]
Suppose $\alpha$, $\sigma$, $\beta$, $\gamma$, and $\delta$ satisfy Assumption \ref{A:lwp}.
Let $u_0 \in |\d_x|^{-\sigma} \hat{M}^{\beta}_{\gamma,\delta}$ and let $u(t)$ be a corresponding solution
given in Theorem \ref{thm:lwp} with maximal lifespan $I_{\max}\ni 0$.
If $T_{\max}<\I$ then $\lim_{T\uparrow T_{\max}}\norm{u}_{S([0,T])}=\I$.
A similar result holds for backward in time.
\end{theorem}

\begin{theorem}[Characterization of scattering]\label{thm:sc}
Suppose $\alpha$, $\sigma$, $\beta$, $\gamma$, and $\delta$ satisfy Assumption \ref{A:lwp}.
Let $u_0 \in |\d_x|^{-\sigma} \hat{M}^{\beta}_{\gamma,\delta}$ and let $u(t)$ be a corresponding solution
given in Theorem \ref{thm:lwp} with maximal lifespan $I_{\max}\ni 0$.
The following three statements are equivalent
\begin{itemize}
\item $u(t)$ scatters forward in time in the sense of Definition \ref{def:scattering};
\item $\norm{u}_{L([0,T_{\max}))} <\I$;
\item $\norm{u}_{S([0,T_{\max}))} <\I$;
\end{itemize}
Further, if either one of the above (hence all of the above) holds then $e^{t\d_x^3}u(t)
$ converges
as $t\to\I$ in $\hat{L}^{\alpha} \cap |\d_x|^{-\sigma} \hat{L}^\beta$.
\end{theorem}
\subsection{Stability estimate}
By a standard argument, we also obtain a stability estimate.
To state it, we introduce a function space with the following norm.
\begin{eqnarray*}
\|f\|_{N(I)}=\||\pt_{x}|^{\frac{1}{2\alpha}}f\|_{L_{x}^{p(N)}(\R; L_{t}^{q(N)}(I))},
\end{eqnarray*}
with
\[
	\(\frac1{p(N)},\frac1{q(N)}\) = \(\frac1{p(M)}, \frac1{q(M)} \) + 2\alpha \(\frac1{p(S)},\frac1{q(S)}  \).
\]

\begin{theorem}\label{thm;stability}
Suppose $\alpha$, $\sigma$, $\beta$, $\gamma$, and $\delta$ satisfy Assumption \ref{A:lwp}.
Let $I\subset \R$ be an interval containing $t_0$. 
Let $\tilde{u}$ be an approximate solution to \eqref{gKdV2} on $I\times\rre$ in such a sense that
\[
	\tilde{u}(t) = e^{-(t-t_0)\d_x^3} \tilde{u}(t_0) + \int_{t_0}^t e^{-(t-s)\d_x^3} (\mu \d_x (|u|^{2\alpha}u) (s)+ e(s) )ds
\]
holds in $L(I) \cap S(I)$ for some function $e \in N(I)$. 
Assume that $\tilde{u}$ satisfies
\begin{eqnarray*}
\|\tilde{u}\|_{S(I)}
+\|\tilde{u}\|_{M(I)}
&\le&M,
\end{eqnarray*}
for some $M>0$. 
Then there exists 
$\varepsilon_{1}=\varepsilon_{1}(M)>0$ 
such that if
\begin{eqnarray*}
\|e^{-t\d_x^3}(u({t_0})-\tilde{u}({t_0}))\|_{S(I)}+
\|e^{-t\d_x^3}(u({t_0})-\tilde{u}({t_0}))\|_{M(I)}+
\|e\|_{N(I)}
\le\varepsilon
\end{eqnarray*}
and $0<\varepsilon<\varepsilon_{1}$, 
then there exists a solution $u$ to (\ref{gKdV}) on 
$I\times\rre$ satisfies 
\begin{align}
\|u-\tilde{u}\|_{S(I)}+
\|u-\tilde{u}\|_{M(I)}
&{}\le C\varepsilon, \label{tt10}\\
\||u|^{2\alpha}u
-|\tilde{u}|^{2\alpha}\tilde{u}
\|_{N(I)}
&{}\le C\varepsilon,\label{tt11}
\end{align}
where the constant $C$ depends only on $M$. 
Further, if $u({t_0})-\widetilde{u}({t_0}) \in |\d_x|^{-\sigma} \hat{M}^{\beta}_{\gamma,\delta}$
for some $\tau \in I$
then, it also holds that
\begin{equation}\label{tt12}
	\||\d_x|^{\sigma}(u-\tilde{u})\|_{L_{t}^{\infty}(I; \hat{M}^{\beta}_{\gamma,\delta} )}
	\le 
	\| |\d_x|^{\sigma}(u(t_{0}) - \tilde{u}(t_{0})) \|_{\hat{M}^{\beta}_{\gamma,\delta}}+
	C\varepsilon.
\end{equation}

\begin{proof}[Proof of Theorem \ref{thm;stability}] 
Once we obtain Theorem \ref{thm:S}, the proof follows from 
the standard continuity argument (See \cite[Lemma 3.1, Proposition 3.2]{MS2}). 
So we omit the detail. 
\end{proof}

\end{theorem}

%
%

\section{Application to a minimizing problem}

\subsection{Linear profile decomposition in $|\d_x|^{-\sigma} \hat{M}^\beta_{2,\delta}$}

In this section, we establish the linear profile decomposition.
The linear profile decomposition essentially consists of two parts.
The first part is concentration compactness and the second part is the inductive procedure to obtain a decomposition.

Let us begin with the concentration compactness part.
The hat-Morrey space $\hat{M}^\alpha_{\beta,\gamma}$ is realized as a dual of a Banach space 
\cite[Theorem 2.17]{M3}.
Therefore, a bounded set of the hat-Morrey space is compact in the weak-$*$ topology.
\begin{theorem}[Concentration compactness in $|\d_x|^{-\sigma} \hat{M}^\beta_{2,\delta}$]\label{thm:cc}
Suppose that $\alpha>8/5$ and $0<\sigma < 1/4-2/(5\alpha)$.
Let $\beta,\gamma,\delta$ satisfy $1/\beta=1/\alpha + \sigma$, 
\[
	\frac4{5\alpha} + 2\sigma < \frac1\gamma < \frac1\beta, \quad\text{and}\quad
	\frac12 -\frac1{5\alpha} < \frac1\delta < \frac1{\beta'}.
\]
Let $\{u_n\}_n \subset |\d_x|^{-\sigma} \hat{M}^\beta_{\gamma,\delta}$ a bounded sequence;
\begin{equation}\label{eq:cc_bdd_asmp}
	\norm{|\d_x|^\sigma u_n}_{\hat{M}^\beta_{\gamma,\delta}} \le M
\end{equation}
for some $M>0$.
If the sequence further satisfies
\begin{equation}\label{eq:cc_additional}
	\norm{ e^{-t\d_x^3} u_n}_{L(\R) \cap S(\R)} \ge m
\end{equation}
for some $m>0$ then there exist
such that
\[
	|\d_x|^{\sigma}(T(y_n)^{-1} A(s_n)^{-1} D(N_n)^{-1} u_n) \rightharpoonup |\d_x|^{\sigma} \psi 
\]
as $n\to\I$ weakly-$*$ in $\hat{M}^{\beta}_{\gamma,\delta}$ with $\norm{\psi}_{\hat{M}^\beta_{\gamma,\delta}} \ge C(M,m)>0$.
\end{theorem}
\begin{proof}
In this proof, all spacetime integrals are taken in $\R\times\R$.
Since the endpoint cases are excluded, 
by means of Theorem \ref{thm:S} and by interpolation inequality, we see that the assumption \eqref{eq:cc_additional}
implies that
$\tnorm{|\d_x|^{1/{3\alpha}} e^{-t\d_x^3} u_n }_{L^{3\alpha}_{t,x}} \ge \tilde{m}$
for some $\tilde{m}=\tilde{m}(\alpha, \sigma,m)>0$.
Let $P_N$ be a standard cut-off operator to $|\xi|\sim N \in 2^\Z$.
We now claim the estimate
\begin{equation}\label{eq:cc_claim}
	\norm{|\d_x|^{\frac1{3\alpha}} e^{-t\d_x^3} u}_{L^{3\alpha}_{t,x}}
	\le C\(\sup_{N\in 2^\Z} 
	\norm{P_N|\d_x|^{\frac1{3\alpha}} e^{-t\d_x^3} u}_{L^{3\alpha}_{t,x}}\)^{1-\frac{\zeta}{3\alpha}}
	\norm{|\d_x|^{\sigma}u}_{\hat{M}^\beta_{\gamma,\delta}}^{\frac{\zeta}{3\alpha}},
\end{equation}
where $\zeta:= \max(\gamma',\delta)$.
By the square function estimate, we have 
\[
	\norm{|\d_x|^{\frac1{3\alpha}}e^{-t\d_x^3}u}_{L^{3\alpha}_{t,x}} \sim
	\norm{\sqrt{ \sum_{N\in 2^\Z} |P_N |\d_x|^{\frac1{3\alpha}}e^{-t\d_x^3}u|^2 }}_{L^{3\alpha}_{t,x}},
\]
We consider only the case $6<3\alpha\le8$, the other cases are similar. As $3\alpha/8\le1$,
\begin{align*}
	\text{(R.H.S of (\ref{eq:cc_claim}))}^{3\alpha}
	&= \iint \prod_{k=1}^4 \(\sum_{N_k\in 2^\Z} |P_{N_k} |\d_x|^{\frac1{3\alpha}}e^{-t\d_x^3}u|^2\)^{\frac{3\alpha}{8}}  dxdt\\
	&{}\le C \sum_{N_1 \le N_2 \le N_3\le N_4} \iint \prod_{k=1}^4 |P_{N_k} |\d_x|^{\frac1{3\alpha}}e^{-t\d_x^3}u|^{\frac{3\alpha}4}  dxdt.
\end{align*}
Let $\eta>0$ be a small number and let $\alpha_1=(\frac1\alpha - \eta)^{-1}$ and $\alpha_4=(\frac1\alpha + \eta)^{-1}$. 
Remark that $2<\zeta < 3\alpha/2$.
It follows from the H\"older inequality that
\begin{align*}
&\iint \prod_{k=1}^4 |P_{N_k} |\d_x|^{\frac1{3\alpha}}e^{-t\d_x^3}u|^{\frac{3\alpha}4}  dxdt \\
&{}\le 
 \(\sup_{N\in2^\Z} \norm{P_{N} |\d_x|^{\frac1{3\alpha}}e^{-t\d_x^3}u}_{L^{3\alpha}_{t,x}}\)^{3\alpha-\zeta}
\prod_{k=1,4} \norm{P_{N_k} |\d_x|^{\frac1{3\alpha}}e^{-t\d_x^3}u}_{L^{3\alpha_{k}}_{t,x}}^{\frac{\zeta}2}.
\end{align*}
By Theorem \ref{thm:S}, we have
\begin{eqnarray}
	\norm{P_{N_k} |\d_x|^{\frac1{3\alpha}}e^{-t\d_x^3}u}_{L^{3\alpha_{j}}_{t,x}}
	&\le&C N_{k}^{\frac{1}{3}(\frac1\alpha-\frac1{\alpha_k})} \norm{ P_{N_j}|\d_x|^{\sigma_{k}}u}_{\hat{M}^{\beta}_{\gamma,\delta}}
	\label{eq:cc_pf1}\\
	&\le& C N_{k}^{\frac{4}{3}(\frac1\alpha-\frac1{\alpha_k})} \norm{ P_{N_j}|\d_x|^{\sigma}u}_{\hat{M}^{\beta}_{\gamma,\delta}}\nonumber
\end{eqnarray}
for $k=1,4$, where $\sigma_k$, $\gamma_k$, $\delta_k$ are chosen by the relations
\[
	\frac{1}{\alpha_k}= \frac1\beta - \sigma_j=\frac1\alpha + (\sigma-\sigma_k), \quad
	\frac1{\gamma_k} = \frac1\beta - \frac1{3\alpha_k} \le \frac1\gamma,\quad
	\frac1{\delta_k} = \frac12 - \frac1{3\alpha_k} \le \frac1\delta.
\]
Remark that $\sigma_1 = \sigma+\eta$ and $\sigma_4 = \sigma-\eta$, and so that
the choice is possible if $\eta>0$ is sufficiently small.
Put $a_{N}:=\norm{ P_{N}|\d_x|^{\sigma}u}_{\hat{M}^{\beta}_{\gamma,\delta}}$ for $N\in 2^\Z$.
Combining these inequalities, we reach to the estimate
\begin{align*}
\norm{|\d_x|^{\frac1{3\alpha}} e^{-t\d_x^3} u}_{L^{3\alpha}_{t,x}}
\le{}& C \(\sup_{N\in2^\Z} \norm{P_{N} |\d_x|^{\frac1{3\alpha}}e^{-t\d_x^3}u}_{L^{3\alpha}_{t,x}}\)^{3\alpha-\zeta}\\
&{}\times\sum_{N_1 \le N_4} a_{N_1}^{\frac{\zeta}2} a_{N_4}^{\frac{\zeta}2} \( \frac{N_1}{N_4}\)^{\frac{4\eta}3}\(1+\log \frac{N_4}{N_1}\)^2.
\end{align*}
Thus, the claim \eqref{eq:cc_claim} follows because
$\norm{a_N}_{\ell_N^\zeta} \le C \norm{|\d_x|^{\sigma}u}_{\hat{M}^{\beta}_{\gamma,\delta}}$
by definition of $\zeta$.

By means of the claim, assumption of the theorem implies that there exists a sequence $\{N_n\} \subset 2^\Z$ such that
\[
	|N_n|^{\frac1{3\alpha}} \norm{P_{N_n} e^{-t\d_x^3}u_n}_{L^{3\alpha}_{t,x}} \ge C(M,m).
\]
As in \eqref{eq:cc_pf1},
\ALN{
	\norm{P_{N_n} e^{-t\d_x^3}u_n}_{L^{3\alpha}_{t,x}}
	&{}\le \norm{P_{N_n} e^{-t\d_x^3}u_n}_{L^{\I}_{t,x}}^{1-\theta}
	\norm{P_{N_n} e^{-t\d_x^3}u_n}_{L^{3\alpha_4}_{t,x}}^{\theta}\\
	&{}\le C \norm{P_{N_n} e^{-t\d_x^3}u_n}_{L^{\I}_{t,x}}^{1-\theta}
	\( M {N_n}^{-\frac1{3\alpha}-\frac{4\eta}3}\)^{\theta},
}
where $\theta = \alpha_4/\alpha = (1+\eta \alpha)^{-1}$.
We obtain
\[
	(N_n)^{-\frac{1}{\alpha}} \norm{e^{-t\d_x^3}P_{N_n} u_n}_{L^{\I}_{t,x}} \ge C(M,m).
\]
Set $v_n(x) := (N_n)^{1/\alpha} u_n(N_nx)$
to obtain $\tnorm{P_{1} e^{-t\d_x^3} v_n}_{L^{\I}_{t,x}} \ge C(M,m)$.
Hence, there exists $(s_n,y_s) \in \R^2$ such that
\EQ{\label{eq:cc_pf2}
	|P_1 e^{{s_n}\d_x^3}v_n|(-y_n) \ge C(M,m).
}
Let $\psi \in |\d_x|^{-\sigma}\hat{M}^\beta_{\gamma,\delta}$ be a 
weak-$*$ limit of $T(-y_n)e^{{s_n}\d_x^3}v_n$ along a subsequence.
Then, by a standard argument, we conclude from \eqref{eq:cc_pf2} that $\norm{|\d_x|^\sigma \psi}_{\hat{M}^\beta_{\gamma,\delta}} \ge \beta(M,m)$.
\end{proof}

We next move to the main issue of this section, linear profile decomposition.
Let us define a set of deformations as follows
\begin{equation}\label{eq:symG}
	G := \{ D(N) A(s) T(y) \ |\ \Gamma = (N, s, y) \in 2^\Z \times
\R\times \R \}.
\end{equation}
We often identify $\mathcal{G} \in G$ with
a corresponding parameter $\Gamma \in 2^\Z \times\R\times \R$
if there is no fear of confusion.
Let us now introduce a notion of orthogonality between two families of deformations.

\begin{definition}\label{def:orthty}
We say two families of deformations $\{\mathcal{G}_n\} \subset G$ and 
$\{ \widetilde{\mathcal{G}}_n \} \subset G$ are \emph{orthogonal} if
corresponding parameters $\Gamma_n ,\widetilde{\Gamma}_n 
\in 2^\Z \times \R\times \R$ satisfies 
\EQ{\label{eq:def_orthty}
	\lim_{n\to\I} \bigg(\abs{\log \frac{N_n}{\widetilde{N}_n}} 
	+ \abs{ s_n -\( \frac{N_n}{\widetilde{N}_n}\)^3\widetilde{s}_n} 
	+ \abs{y_n- \frac{N_n}{\widetilde{N}_n} \widetilde{y}_n } \bigg)= +\I.
}
\end{definition}

\begin{theorem}[Linear profile decomposition in $|\d_x|^{-\sigma}\hat{M}^{\beta}_{2,\delta}$]\label{thm:lpd}
Suppose that $\alpha$, $\sigma$, $\beta$, $\gamma$, and $\delta$ satisfy Assumption \ref{A:min}.
Let $\{u_n\}_n$ be a bounded sequence in $|\d_x|^{-\sigma}\hat{M}^{\beta}_{2,\delta}$.
Then, there exist $\psi^j \in |\d_x|^{-\sigma}\hat{M}^{\beta}_{2,\delta}$,
$r_n^j \in |\d_x|^{-\sigma}\hat{M}^{\beta}_{2,\delta}$, and 
pairwise orthogonal families of deformations $\{\mathcal{G}^j_n\}_n \subset G$ 
($j=1,2,\dots$) parametrized by $\{ \Gamma_n^j = (h_n^j, s_n^j , y_n^j) \}_n$
such that, extracting a subsequence in $n$,
\begin{equation}\label{eq:lpd_decomp}
	u_n = \sum_{j=1}^J \mathcal{G}^j_n \psi^j + r_n^J
\end{equation}
for all $n,J\ge1$ and
\begin{equation}\label{eq:lpd_small}
	\lim_{J\to\I}\varlimsup_{n\to\I} \(
	\norm{ |\d_x|^{\frac1{3\alpha}} e^{-t\d_x^3} r_n^J }_{L^{3\alpha}_{t,x}(\R \times \R)}
	+ \norm{e^{-t\d_x^3} r_n^J }_{L^{\frac{5\alpha}2}_t L^{5\alpha}_{x}(\R \times \R)} \)= 0
\end{equation}
Moreover, a decoupling inequality
\begin{equation}\label{eq:lpd_Pyth}
	\varlimsup_{n\to\I} \norm{|\d_x|^\sigma u_n}_{\hat{M}^{\beta}_{2,\delta}}^\delta 
	\ge \sum_{j=1}^J \norm{|\d_x|^\sigma \psi^j}_{\hat{M}^\beta_{2,\delta}}^\delta 
	+ \varlimsup_{n\to\I} \norm{r_n^J}_{\hat{M}^\beta_{2,\delta}}^\delta 
\end{equation}
holds for all $J\ge1$.
Furthermore, if $u_n$ is real-valued then so are $\psi^j$ and $r_n^J$.
\end{theorem}

\begin{proof}[Proof of Theorem \ref{thm:lpd}] 
For a sequence $\{u_n\}_n\subset |\d_x|^{-\sigma}\hat{M}^{\beta}_{2,\delta}$, define 
\[
	\mathcal{M}(\{u_n\}) := \left\{ \psi \in |\d_x|^{-\sigma}\hat{M}^{\beta}_{2,\sigma} \left|
	\ALNd{
	&\exists \mathcal{G}_n \in G,\, \exists n_k\text{: subsequence s.t.}\\
	&|\d_x|^\sigma \mathcal{G}_{n_k}^{-1} u_{n_k} \rightharpoonup |\d_x|^\sigma \psi \text{ weakly-}*\text{ in }
	\hat{M}^\beta_{2,\delta}.
	}
	\right.\right\}
\]
and $\eta(\{u_n\}) := \sup_{\phi \in \mathcal{M}(\{u_n\})} \tnorm{|\d_x|^\beta \phi}_{\hat{M}^\beta_{2,\delta}}$.
Arguing as in \cite{MS2}, one obtains the desired decomposition
expect that the smallness \eqref{eq:lpd_small} is replaced by
$\lim_{J\to\I} \eta(\{r_n^J\}) = 0$.
However, Theorem \ref{thm:cc} implies that this smallness is a stronger one than \eqref{eq:lpd_small}.
\end{proof}
\begin{remark}
In the previous result \cite{MS2}, decomposition of sequences of real valued functions have a special structure.
However, here it does not.
This is because translation in Fourier side is removed by the boundedness in $|\d_x|^{-\sigma} \hat{M}^\beta_{2,\delta}$.
\end{remark}

\subsection{Proof of Theorem \ref{thm:minimal}}
Let us begin with the analysis of $E_1$.

\begin{proof}[Proof of Theorem \ref{thm:minimal}]
We first take a minimizing sequence
$\{u_n(t), t_n \}_n \subset |\d_x|^{-\sigma} \hat{M}^{\beta}_{2,\delta} \times \R$ as follows;
$t_n \in I_{\max}(u_n)$ and
\begin{eqnarray}
	\tnorm{u_n}_{S([t_n,T_{\max}))} = \I,\quad
	\tnorm{|\d_x|^{\sigma} u_n(t_n)}_{\hat{M}^{\beta}_{2,\delta}} \le E_1 + \frac1n.\label{mins}
\end{eqnarray}
By time translation symmetry, we may suppose that $t_n\equiv0$.
We apply the linear profile decomposition theorem (Theorem \ref{thm:lpd}) 
to the sequence $\{u_n(0)\}_n$.
Then, up to subsequence, we obtain a decomposition
\[
	u_n(0) = \sum_{j=1}^J \mathcal{G}^j_{n} \psi^j + r_n^J
\]
for $n,J\ge1$ with the properties \eqref{eq:lpd_small}, \eqref{eq:lpd_Pyth}, and pairwise orthogonality of $\{\mathcal{G}_n^j\}_n \subset G$.
By extracting subsequence and changing notations if necessary, we may assume that
for each $j$ and $\{x_n^j\}_{n,j} = \{\log N_n^j\}_{n,j}, \{s_n^j\}_{n,j}$, $\{y_n^j\}_{n,j}$,
either $x_n^j \equiv 0$, $x_n^j \to \I$ as $n\to\I$, or 
$x_n^j \to -\I$ as $n\to\I$ holds.
Let us define a nonlinear profile $\Psi^j (t)$ associated with $(\psi^j,s_n^j)$ as follows:
For each $j$, we let
\begin{itemize}
\item if $s_n^j\equiv0$ then $\Psi^j (t)$ is a solution to \eqref{gKdV} with $\Psi^j(0) = \psi^j$;
\item if $s_n^j \to \I$ as $n\to \I$ then $\Psi^j (t)$ is a solution to \eqref{gKdV}
that scatters forward in time to $e^{-t\d_x^3} \\psi^j$;
\item if $s_n^j \to -\I$ as $n\to \I$ then $\Psi^j (t)$ is a solution to \eqref{gKdV}
that scatters backward in time to $e^{-t\d_x^3} \psi^j$;
\end{itemize}
Let 
\EQ{\label{eq:Vnj}
V_n^j(t) := D(N_n^j) T(y_n^j) \Psi^j ((N_n^j)^3t + s_n^{j}).
}
Here, we define an approximate solution
\EQ{
	\tilde{u}_n^J (t,x)= \sum_{j=1}^J V_n^j(t,x) + e^{-t\d_x^3} r_n^J.
	\label{eq:tunj}
}

The main step is to show that there exists $\Psi^j$ that does not scatter forward in time.
Suppose not. Then, all $\Psi^j$ scatters forward in time and so
$\norm{ |\d_x|^\sigma \psi^j}_{\hat{M}^\beta_{2,\delta}} < E_1$ for all $j$.
Then, there exists $M>0$ such that 
\begin{equation}\label{eq:boundKZ}
	\norm{ V_n^j }_{L(\rre_{+})} +
	\norm{ V_n^j }_{S(\rre_{+})} \le M
\end{equation}
holds for any $j,n \ge 1$. 

We shall observe that $\tilde{u}_n^J$ is an approximately solves \eqref{gKdV} and that
is close to $u_n$.
To this end, we provide three intermediate results.

\begin{proposition}[Asymptotic agreement at the initial time]\label{prop:pf:ls1} 
Let $\tilde{u}_{n}^{J}$ be given by \eqref{eq:tunj}. 
Then, it holds for any $J\ge 1$ that
\[
	\norm{\tilde{u}_n^J(0) - u_n(0)}_{\hat{M}^{\alpha}_{2,\sigma}} \to 0
\]
as $n\to\I$.
\end{proposition}
\begin{proposition}[Uniform bound on the approximate solution]\label{prop:pf:ls2}
There exists $M>0$ such that 
\begin{eqnarray}
	\norm{ \tilde{u}_n^J }_{L(\rre_{+})} +
	\norm{ \tilde{u}_n^J }_{S(\rre_{+})}
	\le M\label{u0}
\end{eqnarray}
holds for any $J \ge 1$ and $n \ge N(J)$.
\end{proposition}
\begin{proposition}[Approximate solution to the equation]\label{prop:pf:ls3}
Let $\tilde{u}_n^J$ be defined by \eqref{eq:tunj}. Then 
$\tilde{u}_n^J$ is an approximate solution to \eqref{gKdV}
in such a sense that
\[
	\lim_{J\to\I} 
	\limsup_{n\to\I} \norm{|\d_x|^{-1}[(\d_t + \d_{xxx}) \tilde{u}_n^J -\mu \d_x (|\tilde{u}_n^J|^{2\alpha}\tilde{u}_n^J )]}_{N(\R_+)} = 0.
\]
\end{proposition}
The proof of Proposition \ref{prop:pf:ls1} is obvious by definition of $V_n^j$.
We shall prove Proposition \ref{prop:pf:ls2} later
since our improvement is in this proposition.
Proposition \ref{prop:pf:ls3} follows from Proposition \ref{prop:pf:ls2} and
the mutual asymptotic orthogonality of nonlinear profiles as in \cite[Lemma 4.8]{MS2}.

By means of a stability estimate, 
the above three propositions imply that $\norm{u_n}_{S(\R_+)} < \I$ for sufficiently large $n$.
This contradicts with the definition of $\{u_n\}_n$.

Thus, we see that there exists $j_0$ such that $\Psi^{j_0}$ does not scatter.
Then, $\tnorm{|\d_x|^\sigma \psi^{j_0}}_{\hat{M}^\beta_{2,\delta}} \ge E_1$ by definition of $E_1$.
One also sees from \eqref{eq:lpd_Pyth} that $\tnorm{|\d_x|^\sigma \psi^{j_0}}_{\hat{M}^\beta_{2,\delta}} \le E_1$.
Hence, $\tnorm{|\d_x|^\sigma \psi^{j_0}}_{\hat{M}^\beta_{2,\delta}} = E_1$.

Let us show that $u_{c}:=\Psi^{j_0}$ attains $E_1$.
The case $s_n^{j_0} \to \I$ as $n\to\I$ is excluded since this implies $u_c(t)$ scatters forward in time.
If $s_n^{j_0} \equiv 0$ then $\Psi^{j_0}(0) = \psi^{j_0}$ and so
 $\norm{|\d_x|^\sigma u_c(0)}_{\hat{M}^\beta_{2,\delta}}  = E_1 $.
Finally, if $s_n^{j_0} \to -\I$ as $n\to\I$ then $\lim_{t\to-\I} e^{t\d_x^3} \Psi^{j_0}(t) = \psi^{j_0}$.
Hence, $u_{c,-}:=\lim_{t\to-\I} e^{t\d_x^3}\Psi^{j_0}(t)$ satisfies 
$\norm{|\d_x|^\sigma u_{c,-}}_{\hat{M}^\beta_{2,\delta}}  = E_1 $. 
\end{proof}

To complete the proof of Theorem \ref{thm:minimal},
we prove Proposition \ref{prop:pf:ls2}.
Recall that we have uniform bound \eqref{eq:boundKZ} for each $V_n^j$.
\begin{lemma}\label{lem:pf:small1}
For any $\eps>0$,
there exists $J_0=J_0(\eps)$ such that
\[
	 \norm{ \sum_{j = J_0+1}^{J_0 + k}e^{-t\d_x^3}V_n^j(0) }_{L(\rre_{+})} +
	\norm{ \sum_{j = J_0+1}^{J_0 + k} e^{-t\d_x^3} V_n^j(0) }_{S(\rre_{+})}
 \le \eps
\]
for any $k\ge 1$ and $n\ge N(k)$.
\end{lemma}
\begin{proof}
By definition of $V_n^j(t)$, it suffices to prove the estimate for $e^{-t \d_x^3} \mathcal{G}^j_{n} \psi^j$ instead of $e^{-t\d_x^3} V_n^j(0)$.
By Theorem \ref{thm:S},
\EQ{\label{eq:min_improvepoint}
	\norm{ \sum_{j = J_0+1}^{J_0 + k} e^{-t\d_x^3}
	\mathcal{G}^j_{n} \psi^j}_{L(\rre_{+}) \cap S(\rre_{+})}
	\le C\hMn{ \sum_{j = J_0+1}^{J_0 + k} |\d_x|^\sigma \mathcal{G}^j_{n} \psi^j }{\beta}{2}{\delta}.
}
By pairwise orthogonality of $\{\mathcal{G}_n\}_n$, we see that
\[
\hMn{ \sum_{j = J_0+1}^{J_0 + k} |\d_x|^\sigma \mathcal{G}^j_{n} \psi^j }{\beta}{2}{\delta}
	\le \( \sum_{j = J_0+1}^{J_0 + k} \hMn{ |\d_x|^\sigma \psi^j }{\beta}{2}{\delta}^\delta \)^{1/\delta} + o(1).
\]
By the decoupling inequality \eqref{eq:lpd_Pyth} and the above estimates, we obtain the desired estimate.
\end{proof}
\begin{remark}\label{rem:restriction}
The equation \eqref{eq:min_improvepoint} is the main improvement due to our main theorem.
In the previous result \cite{MS2}, the improved Strichartz estimate is valid only for the diagonal case.
Hence, we use a substitute by interpolating diagonal improved estimate and non-diagonal estimate in $L^\alpha$ space.
The interpolation spoils summability in $j$,
which causes a restriction on possible range of $\alpha$.
\end{remark}

\begin{proof}[Proof of Proposition \ref{prop:pf:ls2}] 
Let $W_n^k:=\sum_{j=J_0+1}^{J_0+k} V_n^j$, 
where $J_{0}$ is fixed later. Then $W_n^k$ satisfies the integral equation
\[
	W_n^k = e^{-t\d_x^3} W_n^k(0)+\mu\int_0^t  e^{-(t-s)\d_x^3} \d_x (|W_n^k|^{2\alpha} W_n^k +  E_n^k) ds
	,
\]
where
$- E_n^k = |W_n^k|^{2\alpha} W_n^k - \sum_{j=J_0+1}^{J_0+k} |V_n^j|^{2\alpha} V_n^j $.
Applying the inhomogeneous Strichartz' estimate \cite[Proposition 2.5]{MS1}, we obtain
\begin{eqnarray}
\label{u1}\\
	\tnorm{W_n^k}_{L(\R_+) \cap S(\R_+)}
	&\le& \tnorm{e^{-t\d_x^3} W_n^k(0)}_{L(\R_+) \cap S(\R_+)}
	+ C \tnorm{W_n^k}_{L(\R_+) \cap S(\R_+)}^{2\alpha+1} \nonumber\\
	& &+ C \tnorm{ E_n}_{N(\R_+)}.\nonumber
\end{eqnarray}
Fix $\eps>0$.
By Lemma \ref{lem:pf:small1}, one can choose $J_0$ so that
\begin{eqnarray}
	\tnorm{e^{-t\d_x^3} W_n^k(0)}_{L(\R_+) \cap S(\R_+)}
	\le \eps\label{u2}
\end{eqnarray}
for any $k\ge 1$ and $n\ge N(k)$. For the above $J_0$, we claim that
\begin{eqnarray}
\norm{ E_n}_{N(\R_+)} \le \eps\label{u3}
\end{eqnarray}
for any $k \ge 1$ and $n\ge N(k)$. 
By the interpolation inequality and the Young inequality,
\begin{align*}
	\tnorm{ E_n}_{N(\R_+)} \le
	{}&\tnorm{|\d_x|^{1/\alpha}E_n}_{L^{p_1}_xL^{q_1}_t(\R\times \R_+)}^{\frac12} 
	\tnorm{ E_n}_{L^{p_2}_xL^{q_2}_t(\R\times \R_+)}^{\frac12},
\end{align*}
where $1/p_1=1/p(L) + 2\alpha/p(S)$, $1/q_1=1/q(L) + 2\alpha/q(S)$, $p_2=p(S)/(2\alpha+1)$, and $q_2=q(S)/(2\alpha+1)$.
We have
\[
	\tnorm{|\d_x|^{1/\alpha} E_n}_{L^{p_{1}}_x L^{q_{1}}_t (\R \times \R_+)}
	\le C \sum_{j=J_0+1}^{J_0+k} \norm{V_n^j}_{L(\R_+)\cap S(\R_+)}^{2\alpha+1}
	\le Ck M^{2\alpha+1}
\]
in light of \eqref{eq:boundKZ}. On the other hand, for any $k$, we have
$\tnorm{ E_n}_{L^{p_2}_xL^{q_2}_t(\R\times \R_+)} \to 0$ as $n\to\I$ (see \cite[Lemma 4.8]{MS2}).
Thus, we obtain (\ref{u3}) for any $k\ge 1$ and $n\ge N(k)$. 
Combining (\ref{u1}), (\ref{u2}), (\ref{u3}) and the 
continuity argument, we have that if $\varepsilon$ is sufficiently small, then 
$\tnorm{W_n^k}_{L(\R_+) \cap S(\R_+)}\le C\varepsilon$ 
for any $k\ge 1$ and $n\ge N(k)$. Combining this with (\ref{eq:lpd_small}), 
we obtain the uniform estimate (\ref{u0}).
\end{proof}

\subsection{Proof of Theorem \ref{thm:minimal2}}
We finally consider analysis of $E_2$.
\begin{proof}
By definition of $E_2$, it is possible to choose a minimizing sequence of solutions $\{u_n(t)\}_n$ so that
all $u_n(t)$ does not scatter forward in time and
\[
	E_2 \le \varlimsup_{t\uparrow T_{\max}(u_n)} \norm{|\d_x|^\sigma u_n(t)}_{\hat{M}^\beta_{2,\delta}} \le E_2 + \frac1n.
\] 
Hence, there exists $t_n,t_n' \in I_{\max}(u_n)$, $t_n< t_n'$, so that
\[
	\norm{u_n}_{S([t_n,t_n'])}\ge n,\quad \sup_{t\in [ t_n,T_{\max})} \norm{|\d_x|^\sigma u_n(t)}_{\hat{M}^\beta_{2,\delta}} \in \left[E_2,E_2+\frac2n\right].
\]
Indeed, we first choose $t_n$ so that the second property holds.
Then, since $\norm{u_n}_{S([t_n,T_{\max}))}=\I$,
we can choose $t_n'$ so that the first property is true.

By time translation symmetry, we may suppose that $t_n'\equiv0$.
We now apply linear profile decomposition to $u_n(0)$ to get the decomposition
\[
	u_n(0) = \sum_{j=1}^J \mathcal{G}^j_{n} \psi^j + r_n^J
\]
for $n,J\ge1$ with the properties \eqref{eq:lpd_small}, \eqref{eq:lpd_Pyth}, and pairwise orthogonality of $\{\mathcal{G}_n^j\}_n \subset G$.
By extracting subsequence and changing notations if necessary, we may assume that
for each $j$ and $\{x_n^j\}_{n,j} = \{\log N_n^j\}_{n,j}, \{s_n^j\}_{n,j}$, $\{y_n^j\}_{n,j}$,
we have either $x_n^j \equiv 0$, $x_n^j \to \I$ as $n\to\I$, or 
$x_n^j \to -\I$ as $n\to\I$.
Let us define nonlinear profile $\Psi^j$ associated with $(\psi^j,s_n^j)$ in the same way as in the proof of Theorem \ref{thm:minimal}.
We also define $V_n^j$ and $\tilde{u}_n^J$ by \eqref{eq:Vnj} and \eqref{eq:tunj}, respectively.

Then, mimicking the proof of Theorem \ref{thm:minimal},
one sees that at least one $\Psi^j$ does not scatter forward in time.
We further see from decoupling inequality \eqref{eq:lpd_Pyth} and small data scattering that
the number of the profiles that do not scatter is finite.
Renumbering, we may suppose that $\Psi^j(t)$ do not scatter forward in time if and only if $j\in [1,J_1]$.
Here, $1\le J_1 < \I$.
Arguing as in \cite{M3}, we see that $J_1=1$, $\varlimsup_{t\uparrow T_{\max}(\Psi^1)}\tnorm{|\d_x|^\sigma \Psi^1(t)}_{\hat{M}^\beta_{2,\delta}}=E_2$, $\psi^j \equiv 0$ for $j\ge2$, and $r_n^1\to0$ as $n\to\I$ in $|\d_x|^{-\sigma}
\hat{M}^\beta_{2,\delta}$.
As a result,
\EQ{\label{eq:min2_pf1}
	u_n(0) = \mathcal{G}_n^1 \psi^1 + o_n(1)\IN |\d_x|^{-\sigma} \hat{M}^\beta_{2,\delta}.
}
If $s_n^1\to\I$ as $n\to\I$ then $\Psi^1(t)$ scatters forward in time, a contradiction.
Because of $\norm{u_n}_{S([t_n,0])}\ge n$, the same argument works for negative time direction.
We see that $\Psi^1(t)$ does not scatter backward in time and that the case $s_n^1\to-\I$ as $n\to\I$ is excluded.
Moreover, together with $\sup_{t\in [ t_n,T_{\max})} \norm{|\d_x|^\sigma u_n(t)}_{\hat{M}^\beta_{2,\delta}} \in \left[E_2,E_2+\frac2n\right]$, we have
\[
	\varlimsup_{t\downarrow T_{\min}(\Psi^1)}\norm{|\d_x|^\sigma \Psi^1(t)}_{\hat{M}^\beta_{2,\delta}}=
	\sup_{t\in I_{\max}(\Psi^1)} \norm{|\d_x|^\sigma \Psi^1(t)}_{\hat{M}^\beta_{2,\delta}} = E_2.
\]

So far, we have proven that $\Psi^1$ satisfies the first two properties of Theorem \ref{thm:minimal2}.
Let us finally prove the precompactness modulo symmetry.
Take an arbitrary sequence $\{\tau_n\} \subset I_{\max}(\Psi^1)$.
Then, we can choose $t_n \in (T_{\min}(\Psi^1),\tau_n)$ so that $u_n(t):=\Psi$, $t_n'=\tau_n$, and this $t_n$
satisfies the same assumption as above.
The decomposition \eqref{eq:min2_pf1} reads as existence of $\psi \in |\d_x|^{-\sigma}\hat{M}^{\beta}_{2,\delta}$,
$\{N_n\}_n \subset \R_+$, and $\{y_n\}_n \subset \R$ such that
\[
	\Psi^1(\tau_n) = D(N_n) T(y_n) \phi + o_n(1)\IN |\d_x|^{-\sigma} \hat{M}^\beta_{2,\delta}.
\]
This is nothing but a sequential version of precompactness.
A standard argument then upgrades this property to the continuous one.
\end{proof}

%
%

\appendix

\section{Embedding in the generalized Morrey space}

In this appendix we mention the embedding properties 
of the generalized Morrey space. 
We first note that $M^{\beta}_{\gamma,\I}$ is a usual Morrey space. 
We easily see that $M^{\beta}_{\beta,\I}=L^{\beta}$ with equal norm. 

We collect the inclusion relations for 
the generalized Morrey space.

\vskip1mm
\noindent
(i) For any $1 \le \gamma_2 \le \gamma_1 \le \beta \le \I$ 
and $1\le \delta_1 \le \delta_2 \le \I$,
it holds that $M^{\beta}_{\gamma_1,\delta_1} \hookrightarrow 
M^{\beta}_{\gamma_2,\delta_2}$.

\vskip1mm
\noindent
(ii) For any $1 \le \beta \le \gamma_1 \le \gamma_2 
\le \I$ and $1\le \delta_1 \le \delta_2 \le \I$,
it holds that $\hat{M}^{\beta}_{\gamma_1,\delta_1} \hookrightarrow \hat{M}^{\beta}_{\gamma_2,\delta_2}$

\vskip1mm
\noindent
(iii) $L^{\beta} \hookrightarrow M^{\beta}_{\gamma,\delta}$ 
holds as long as 
$1\le \gamma<\beta<\delta\le\I$.

\vskip1mm
\noindent
(iv) $\hat{L}^{\beta} \hookrightarrow 
\hat{M}^{\beta}_{\gamma,\delta}$ holds as long as 
$ 1\le  \gamma'<\beta'<\delta \le\I$.

\vskip1mm
\noindent
(v) $|\d_x|^{-\si} \hat{M}^{\beta}_{\gamma_1,\delta_1} \hookrightarrow  \hat{M}^{\alpha}_{\gamma_2,\delta_2}$ 
if $1/\gamma_1-1/\gamma_2>1/\beta-1/\alpha= \si$ and $\delta_1 \le \delta_2$.

\vskip1mm 
The properties (i) and (ii) are trivial from the definition of 
the generalized Morrey space. 
For the proof of (iii) and (iv), see \cite[Proposition A.1]{MS2}. 
We now give the proof of (v). The H\"older inequality yields 
\[
	\norm{\hat{f} }_{L^{\gamma'_2}(\tau^{j}_k)}
	\le \norm{|\xi|^{-\si}}_{L^{\frac{\gamma_1\gamma_2}{\gamma_2-\gamma_1}}(\tau^j_k)} 
	\norm{|\xi|^{\si} \hat{f} }_{L^{\gamma'_1}(\tau^{j}_k)}
	= C |\tau_k^j|^{\frac1{\gamma_1} -\frac1{\gamma_2}-\si}  \norm{|\xi|^{\si}  \hat{f} }_{L^{\gamma'_1}(\tau^{j}_k)}
\]
if $1/\gamma_1-1/\gamma_2-\si>0$.
Hence, it follows that
\begin{align*}
	\norm{f}_{\hat{M}^{\alpha}_{\gamma_2,\delta_2}}
	&{}= \norm{|\tau_k^j|^{\frac1{\gamma_2}-\frac1{\alpha}} 
	\norm{\hat{f} }_{L^{\gamma'_2}(\tau^{j}_k)}}_{\ell^{\delta_2}_{k,j}}\\
	&{}\le C \norm{|\tau_k^j|^{\frac1{\gamma_1}-\si-\frac1{\alpha}} \norm{|\xi|^{\si}  \hat{f} }_{L^{\gamma'_1}(\tau^{j}_k)}}_{\ell^{\delta_1}_{k,j}} \le C \norm{ |\d_x|^{\si} f }_{\hat{M}^{\beta}_{\gamma_1,\delta_1}}
\end{align*}
as long as $1/\gamma_1-1/\gamma_2>\si=1/\beta-1/\alpha$ and $\delta_1 \le \delta_2$. 
Hence we have (v).

\vskip3mm
\noindent {\bf Acknowledgments.} 
S.M. is partially supported by the Sumitomo Foundation, Basic Science Research
Projects No.\ 161145. 
J.S. is partially supported by JSPS,
Grant-in-Aid for Young Scientists (A) 25707004.

\end{document}